\documentclass[11pt,a4paper]{article}
\usepackage{amsmath,amssymb,latexsym,amsthm}  
\usepackage[english]{babel}
\usepackage[latin2]{inputenc}

\newtheorem{Theo}{Theorem}[section]
\newtheorem{Lemma}[Theo]{Lemma}
\newtheorem{Prop}[Theo]{Proposition}

\theoremstyle{definition}

\numberwithin{equation}{section}

\begin{document}

\title{On the multiplication groups of three-dimensional topological loops}
\author{\'Agota Figula} 
\date{}
\maketitle

\begin{abstract}
We clarify the structure of nilpotent Lie groups which are multiplication groups of $3$-dimensional simply connected  topological loops and prove that non-solvable Lie groups acting minimally on $3$-dimensional manifolds cannot be the multiplication group of $3$-dimensional topological loops. Among the nilpotent Lie groups for any filiform groups ${\mathcal F}_{n+2}$ and ${\mathcal F}_{m+2}$ with 
$n, m > 1$, the direct product ${\mathcal F}_{n+2} \times \mathbb R$ and the direct product ${\mathcal F}_{n+2} \times _Z {\mathcal F}_{m+2}$ with amalgamated center $Z$ 
 occur as the multiplication group of $3$-dimensional topological loops. To obtain this result we classify all 
 $3$-dimensional simply connected topological loops having a $4$-dimensional nilpotent Lie group as the group topologically generated by the left translations.  
\end{abstract}



\noindent
{\footnotesize {{\em Thanks: } The work of the first author was supported by the Hungarian Scientific Research Fund (OTKA) Grant PD 77392 and by the EEA and Norway Grants (Zolt\'an Magyary Higher Education Public Foundation).} }

\section{Introduction}

The multiplication group $Mult(L)$ and the inner mapping group $Inn(L)$ of a loop $L$ 
introduced in \cite{albert1}, \cite{bruck}, are important tools for research in loop theory since they reflect strongly the structure of $L$. In particular, there is a strict correspondence between the normal subloops of $L$ and certain normal subgroups of $Mult(L)$.   
Hence, it is an interesting question which groups can be represented as multiplication groups  of loops (\cite{kepka}, \cite{kepka2},  \cite{vesanen2}). A purely group theoretical characterization of multiplication groups is given in \cite{kepka}. 

The mainly studied topological loops $L$ are those which are realized as
sharply transitive sections in Lie groups $G$ such that the left translations
of $L$ generate $G$ (cf. \cite{loops}, \cite{figula1}, \cite{figula2}). For most of these loops the group $Mult(L)$ 
generated by all left and right translations has infinite dimension. In \cite{figula} it is
shown that the condition for $Mult(L)$ to be a Lie group is a strong restriction. Namely, for $2$-dimensional topological loops $L$ the group $Mult(L)$ is a
Lie group if and only if it is an elementary filiform Lie group of dimension $\ge 4$. In this paper we show that for $3$-dimensional loops for which the group $Mult(L)$ is nilpotent the situation changes radically. Namely, there is no
topological loop $L$ having the $4$-dimensional filiform Lie group as the group
$Mult(L)$ (cf. Proposition \ref{filiform4dim}), but there is a plethora of loops $L$ for which $Mult(L)$ is a nilpotent
but not filiform Lie group (cf. Propositions \ref{directproduct}, \ref{directproductsfiliform}). In contrast to this a simple proof shows that non-solvable Lie groups acting minimally on $3$-dimensional manifolds cannot occur as the multiplication group of $3$-dimensional topological loops (cf. Theorem \ref{minimal}). A $3$-dimensional connected simply connected topological loop $L$ such that $Mult(L)$ is a Lie group is homeomorphic either to the $3$-sphere $S^3$ or to the affine space $\mathbb R^3$. We prove that if $L$ is  quasi-simple and homeomorphic to $S^3$, then the group $Mult(L)$ is either quasi-simple or isomorphic to 
$Spin_3 \rtimes SO_3(\mathbb R)$. If a quasi-simple connected topological loop $L$ is homeomorphic to $\mathbb R^3$ and its 
multiplication group $Mult(L)$ is not quasi-simple, then one of the following holds: 
If $Mult(L)$ is semi-simple, then it is the semidirect product $\widetilde{PSL_2(\mathbb R)} \rtimes PSL_2(\mathbb R)$, where 
the universal covering $\widetilde{PSL_2(\mathbb R)}$ of $PSL_2(\mathbb R)$ is homeomorphic to $\mathbb R^3$. 
If $Mult(L)$ is not semi-simple, then it is the semidirect product $\mathbb R^3 \rtimes S$, where $S$ is isomorphic either to the group $Spin_3(\mathbb R)$ or to $SL_3(\mathbb R)$ respectively to $PSL_2(\mathbb R)$ (cf. Proposition \ref{nonsolvable}).

As applications of Proposition \ref{2dimensionalcentre} first we determine all $3$-dimensional simply connected 
topological loops $L$ having a $4$-dimensional nilpotent Lie group   
as the group topologically generated by their left translations. The loop multiplications in these cases are uniquely determined by one continuous real function of one or two variables. The multiplication groups of these loops are Lie groups precisely if the continuous function in the loop multiplication derives from exponential polynomials. 
If $L$ has $2$-dimensional centre, then we show that the groups $\mathbb R \times {\mathcal F}_{n+2}$,
$n \ge 2$, where ${\mathcal F}_{n+2}$ is the $n+2$-dimensional elementary filiform Lie group, are multiplication groups of $L$ 
(cf. Proposition \ref{directproduct}). Moreover, among the
at most $5$-dimensional nilpotent Lie groups only the group $\mathbb R \times {\mathcal F}_4$ occurs as the multiplication group of $3$-dimensional topological loops with $2$-dimensional centre.  
Also the direct products of two elementary filiform Lie groups with amalgamated centers are multiplication groups of $L$. These loops $L$ have $1$-dimensional centre $Z(L)$ such that $L/Z(L)$ is the
abelian group $\mathbb R^2$ (cf. Proposition \ref{directproductsfiliform}).

\section{Preliminaries}

A binary system $(L, \cdot )$ is called a loop if there exists an element 
$e \in L$ such that $x=e \cdot x=x \cdot e$ holds for all $x \in L$ and the 
equations $a \cdot y=b$ and $x \cdot a=b$ for given $a,b \in L$ have precisely one solution which we denote 
by $y=a \backslash b$ and $x=b/a$. 
\newline
The left and right translations $\lambda _a= y \mapsto a \cdot y :L \to L$ and 
$\rho _a: y \mapsto y \cdot a:L \to L$, $a \in L$,  are permutations of $L$. 
\newline 
The permutation group $Mult(L)=\langle \lambda_a, \rho_a; \ a \in L \rangle $ is called the multiplication group of $L$. 
The stabilizer of the identity element $e \in L$ in $Mult(L)$ is denoted by $Inn(L)$, and  $Inn(L)$ is called  the inner mapping group of $L$. 

Let $K$ be a group, $S \le K$, and let $A$ and $B$ be two left transversals to $S$ in $K$ (i.e. two systems of representatives for the left cosets of the subgroup $S$ in $K$). We say that $A$ and $B$ are $S$-connected if $a^{-1} b^{-1} a b \in S$ for every $a \in A$ and $b \in B$. 
By $C_K(S)$ we denote the core of $S$ in $K$ (the largest normal subgroup of $K$ contained in $S$). If $L$ is a loop, then $\Lambda (L)=\{ \lambda _a; \ a \in L \}$ and $R (L)=\{ \rho _a; \ a \in L \}$ are $Inn(L)$-connected transversals in the group $Mult(L)$ and the core of $Inn(L)$ in $Mult(L)$ is trivial. The connection between multiplication groups of loops and transversals is given in \cite{kepka} by Theorem 4.1.  This theorem yields the following 

\begin{Lemma} \label{kepkalemma} Let $L$ be a loop and $\Lambda (L)$ be the set of left translations 
of $L$. Let $K$ be a group containing $\Lambda (L)$ and $S$ be a  subgroup of $G$ with $C_K(S)=1$ 
such that $\Lambda (L)$ is a left transversal to $S$ in $K$. The group $K$ is isomorphic to the 
multiplication group 
$Mult(L)$ of $L$ if and only if  there is a left transversal $T$ to $S$ in $K$ such that $\Lambda (L)$ and $T$ are $S$-connected and $K=\langle \Lambda (L), T \rangle $. In this case $S$ is isomorphic to the inner mapping group $Inn(L)$ of $L$.  
\end{Lemma}

\noindent
The kernel of a homomorphism $\alpha :(L, \cdot ) \to (L', \ast )$ of a loop $L$ into a loop $L'$ is a normal subloop $N$ of $L$, i.e. a subloop of $L$ such that
\[ x \cdot N=N \cdot x, \ \ (x \cdot N) \cdot y= x \cdot (N \cdot y), \ \ x \cdot ( y \cdot N)=(x \cdot y) \cdot N. \] 

The centre $Z(L)$ of a loop $L$ consists of all elements $z$ which satisfy the equations $z x \cdot y=z \cdot x y, \ x \cdot y z=x y \cdot z, \ x z \cdot y=x \cdot z y, \ z x =x z$ for  all $x,y \in L$.
If we put $Z_0=e$, $Z_1=Z(L)$ and $Z_i/Z_{i-1}=Z(L/Z_{i-1})$, then we obtain a series of normal subloops of $L$. If $Z_{n-1}$ is a proper subloop of $L$ but $Z_n=L$, then $L$ is centrally nilpotent of class $n$.
In \cite{bruck} it was proved that the nilpotency of the multiplication group $Mult(L)$ of $L$ implies that $L$ is centrally nilpotent. 

\begin{Lemma} \label{brucklemma} Let $L$ be a loop with multiplication group $Mult(L)$ and identity element $e$. 
\newline
\noindent
(i) Let $\alpha $ be a homomorphism of the loop $L$ onto the loop $\alpha (L)$ with kernel $N$. Then $N$ is a normal subloop  of $L$ and $\alpha $ induces a homomorphism of the group $Mult(L)$ onto the group $Mult(\alpha (L))$. 

Denote by $M(N)$ the set $\{ m \in Mult(L); \ x N=m(x) N \ \hbox{for  all} \ x \in L \}$. 
Then $M(N)$ is a normal subgroup of $Mult(L)$ containing the multiplication group $Mult(N)$ of the loop $N$, the multiplication group of the factor loop $L/N$ is isomorphic to $Mult(L)/M(N)$.  
\newline 
\noindent 
(ii) For every normal subgroup $\mathcal{N}$ of $Mult(L)$ the orbit $\mathcal{N}(e)$ is a normal subloop of $L$. Moreover, $\mathcal{N} \le M(\mathcal{N}(e))$.  
\end{Lemma} 
\begin{proof} The assertion of (i) is proved by A. Albert in \cite{albert1} (Theorems 3, 4 and 5, pp. 513-515). The assertion (ii) is proved by Bruck in \cite{bruck2}, p. 62.  
\end{proof} 

\bigskip
\noindent
A loop $L$ is called topological if $L$ is a topological space and the binary operations $(x, y) \mapsto x \cdot y$, $(x, y) \mapsto  x \backslash y, (x,y) \mapsto y/x :L \times L \to L$  are continuous.
Every connected topological loop $L$ having a Lie group  as the group topologically generated by the left translations is obtained on a homogeneous space $G/H$, where $G$ is a connected Lie group, $H$ is a closed subgroup containing no non-trivial normal subgroup of $G$ and $\sigma :G/H \to G$ is  a continuous sharply transitive section with
$\sigma (H)=1 \in G$ such that the subset $\sigma (G/H)$ generates $G$.  The multiplication  of  $L$ on the manifold  $G/H$  is  defined by
$x H \ast y H=\sigma (x H) y H$ and  the group $G$ is  the group topologically generated by the left translations of $L$. Moreover, the subgroup $H$ is the stabilizer of the identity element
$e \in L$ in the group $G$.

A quasi-simple connected Lie group is a connected Lie group $G$ such
that any normal subgroup of $G$ is discrete and central in $G$. A semi-simple
connected Lie group $G$ has the form $G=G_1 \cdot G_2 \cdots G_r$, where $G_i$ are normal quasi-simple connected Lie 
subgroups such that $G_i \cap G_j$ is a discrete central subgroup of $G$.  
A connected loop $L$ is quasi-simple if any normal subloop of $L$ is discrete in
$L$. According to \cite{hofmann2}, p. 216, all discrete normal subloops of a connected loop
are central.

The elementary filiform Lie group ${\mathcal F}_{n+2}$ is the 
 simply connected Lie group of dimension $n+2 \ge 3$ whose Lie algebra 
is elementary filiform, i.e. it has a basis $\{ e_1, \cdots , e_{n+2} \}$  such that  $[e_1, e_i]= e_{i+1}$ for 
$2 \le i \le n+1$ and all other Lie brackets are zero.  A $2$-dimensional simply connected loop $L_{\mathcal F}$ is called 
an elementary filiform loop if its multiplication group is an elementary filiform group  ${\mathcal F}_{n+2}$, $n \ge 2$.   

A foliation is a decomposition of a manifold as a union of disjoint submanifolds of smaller dimension.

\section{On the multiplication group of $3$-dimensional topological loops}

Let $L$ be a topological loop on a connected $3$-dimensional manifold such that the group $Mult(L)$ topologically generated by all left and right translations of $L$ is a Lie group. The loop $L$ itself is a $3$-dimensional homogeneous space with respect to the transformation group $Mult(L)$. 
Since there does not exist a multiplication with identity on the sphere $S^2$ a simply connected $3$-dimensional topological loop $L$ having a Lie group as its multiplication group is homeomorphic to $\mathbb R^3$ or to $S^3$ (see \cite{gorbatsevich}, p. 210). 

A transitive action of a connected Lie group $G$ on a manifold $M$ is called minimal, if it is locally effective and if $G$ does not contain subgroups acting transitively on $M$. The minimal actions of non-solvable Lie groups on $3$-dimensional manifolds are given in \cite{gorbatsevich}, Table 1, p. 201. 

\begin{Theo} \label{minimal} There does not exist $3$-dimensional proper connected topological loop $L$ such that its multiplication group $Mult(L)$ acts minimally on the manifold $L$ and $Mult(L)$ is a non-solvable Lie group.  
\end{Theo} 
\begin{proof} We may assume that $L$ is simply connected and hence it is homeomorphic to $\mathbb R^3$ or to $S^3$. As 
$dim (Mult(L)) \ge 4$ it follows from \cite{gorbatsevich}, p. 201, that $L$ is homeomorphic to $\mathbb R^3$ and the radical $\mathcal{R}$ of the group $Mult(L)$ has positive dimension. As $Mult(L)$ acts effectively and minimal transitively on the manifold $L$ the orbit $\mathcal{R}(e)$ has dimension $1$ or $2$. Since $\mathcal{R}(e)$ is a normal subloop of $L$ 
(cf. Lemma \ref{brucklemma}) the factor loop $\mathcal{F}=L/ \mathcal{R}(e)$ is a connected loop of dimension $1$ or $2$. The multiplication group 
$Mult( \mathcal{F})$ of $\mathcal{F}$ is a factor group $Mult(L)/M$, where $M \neq Mult(L)$ is a connected normal subgroup of  $Mult(L)$ containing $\mathcal{R}$ (cf. Lemma \ref{brucklemma}). Hence $Mult(\mathcal{F})$ is a semi-simple Lie group. 
Since every at most $2$-dimensional connected loop having a Lie group as its multiplication group is either a connected Lie group or an elementary filiform loop (cf. Lemma 18.18 in \cite{loops}, p. 248, and Theorem 1 in \cite{figula}) we obtain a contradiction to the fact that the multiplication groups of these loops are solvable.  \end{proof}

\medskip
\noindent 
A transitive action of a Lie group $G$ on a manifold $M$ is called primitive, if on $M$ there is no $G$-invariant foliation with connected fibres of positive dimension smaller than $\hbox{dim} \ M$.

\begin{Prop} \label{nonsolvable} Let $L$ be a $3$-dimensional quasi-simple connected simply connected topological loop such that the multiplication group $Mult(L)$ of $L$ is a Lie group. 
\newline
\noindent
(a) If $L$ is homeomorphic to $S^3$, then the group $Mult(L)$ is either quasi-simple or isomorphic to the semidirect product 
$Spin_3(\mathbb R) \rtimes SO_3(\mathbb R)$.  
\newline
\noindent
(b) If $L$ is homeomorphic to $\mathbb R^3$ and the group $Mult(L)$ is not quasi-simple, then one of the following holds: 
\newline
\noindent 
(i) If $Mult(L)$ is semi-simple, then it is isomorphic to the semidirect product 
$\widetilde{PSL_2(\mathbb R)} \rtimes PSL_2(\mathbb R)$.  
\newline
\noindent
(ii) If $Mult(L)$ is not semi-simple, then it is the semidirect product $\mathbb R^3 \rtimes S$, where $S$ is isomorphic either to $Spin_3(\mathbb R)$ or to $SL_3(\mathbb R)$ respectively  to $PSL_2(\mathbb R)$   and acts irreducibly on $\mathbb R^3$. 
\end{Prop} 
\begin{proof} Let $L$ be a $3$-dimensional quasi-simple connected topological loop such that the group $Mult(L)$ is a Lie group. 
Then the Lie group $Mult(L)$ acts primitively on $L$ and by Lemma \ref{brucklemma} (ii) for every non-trivial connected normal subgroup $\mathcal N$ of $Mult(L)$ the orbit $\mathcal N(e)$ is a normal subloop of $L$. For 
$\mathcal N(e)= \{e \}$ the inner mapping group $Inn(L)$ contains the normal subgroup $\mathcal N$ of $Mult(L)$ which is a contradiction. Therefore $\mathcal N(e)$ is the whole loop $L$. Hence every non-trivial connected normal subgroup $\mathcal N$ of $Mult(L)$ operates transitively on $L$. Every solvable Lie group has a one- or two-dimensional connected normal subgroup $K$. Since $K$ cannot act transitively on $L$ the Lie group $Mult(L)$ cannot be solvable. 

We may assume that $L$ is simply connected, since otherwise we would consider the universal covering of $L$. Then the loop $L$ is homeomorphic to $\mathbb R^3$ or to $S^3$.  
If the group $Mult(L)$ is semi-simple, then it has the form $Mult(L)=G_1 \cdot G_2 \cdots G_r$, where $G_i$ are normal 
quasi-simple connected Lie subgroups such that $G_i \cap G_j$ is a discrete central subgroup of $Mult(L)$. Hence every $G_i$ acts transitively on $L$. If there is a proper normal subgroup $S$ of $Mult(L)$ such that the stabilizer of $e \in L$ in the group $S$ is a non-trivial connected subgroup $S_e$ of $S$ and there is a subgroup $G_j$ such that $S \cap G_j$ is a discrete 
central subgroup of $Mult(L)$, then for all $g \in G_j$ and $\alpha \in S_e$ we have $g(e)=g(\alpha (e))= \alpha (g(e))$.  
Since $G_j$ acts transitively on $L$ the group $S_e$ fixes every element of the loop $L$. Hence the action of the group $Mult(L)$ on $L$ is not effective. From this contradiction it follows that $r=2$ and the subgroups $G_1$ and $G_2$ act 
sharply transitively on $L$. Hence $G_1$ and $G_2$ have dimension $3$. As $L$ is simply connected $G_1$ as well as $G_2$ are homeomorphic to $\mathbb R^3$ or to $S^3$. It follows that $G_1$ is isomorphic to $\widetilde{PSL_2(\mathbb R)}$ or to $Spin_3(\mathbb R)$ and the group $Mult(L)$ is the semidirect product $G_1 \rtimes G_e$, where the stabilizer $G_e$ of 
$e \in L$ in $Mult(L)$ is a $3$-dimensional automorphism group acting faithfully on $G_1$.

Now we assume that the radical $\mathcal R$ of the non-solvable group $Mult(L)$ is a non-trivial connected normal subgroup of $Mult(L)$.  As $Mult(L)$ acts primitively and effectively on $L$ the group $Inn(L)$ is a maximal subgroup of $Mult(L)$, the orbit $\mathcal R(e)$ is the whole loop $L$ and  $\hbox{dim} \mathcal R =3$. Since the commutator subgroup $\mathcal R '$ of $\mathcal R$ is normal in the group $Mult(L)$ it must be trivial otherwise we have a contradiction to the fact that $Inn(L)$ is maximal in $Mult(L)$.  Therefore 
$\mathcal R$ is a $3$-dimensional commutative connected normal subgroup of $Mult(L)$ acting sharply transitively on $L$. Then the simply connected loop $L$ is homeomorphic to $\mathbb R^3$, we have $\mathcal R \cong \mathbb R^3$ and 
$Mult(L)= \mathbb R^3 \rtimes S$, where $S$ is a semi-simple group of automorphisms of $\mathbb R^3$. 
Therefore $S$ is isomorphic either to $Spin_3(\mathbb R)$ or to $SL_3(\mathbb R)$ respectively to $PSL_2(\mathbb R)$.  
\end{proof}

\begin{Lemma} \label{simplyconnected} Let $L$ be a $3$-dimensional proper connected topological loop having a solvable Lie group as the multiplication group $Mult(L)$ of $L$. If $L$ is simply connected, then it is homeomorphic to $\mathbb R^3$. 
\end{Lemma}
\begin{proof} By Theorem 3.2 in \cite{gorbatsevich}, p. 208, the sphere $S^3$ is not a solvmanifold.  Hence $L$ cannot be homeomorphic to $S^3$ and the assertion follows.  \end{proof}

\medskip
\noindent  
Now we deal with the case that the multiplication group $Mult(L)$ of the loop $L$ is nilpotent.

\begin{Lemma} Let $L$ be a $3$-dimensional proper connected simply connected topological loop such that its multiplication group 
$Mult(L)$ is a nilpotent Lie group. Then the loop $L$ is centrally nilpotent. The loop $L$ is an extension of a $2$-dimensional centrally nilpotent loop $M$ by the abelian Lie group $\mathbb R$ and also an extension of the abelian Lie group $\mathbb R$ by a $2$-dimensional centrally nilpotent loop $M$. 
\end{Lemma} 
\begin{proof} 
By Lemma \ref{simplyconnected} the loop $L$ is homeomorphic to $\mathbb R^3$ and by Proposition \ref{nonsolvable} it has a proper connected normal subloop. Hence the multiplication group $Mult(L)$ acts imprimitively on $L$. 
One can distinguish three classes of Lie groups acting imprimitively on $\mathbb R^3$ (cf. \cite{lie}, pp. 141-178): 
\newline
\noindent
I. Lie groups $G$, such that on $\mathbb R^3$ there is a $G$-invariant foliation with $2$-dimensional connected fibres $P$, but there is no $G$-invariant foliation of $P$ with $1$-dimensional connected fibres. 
\newline
\noindent
II. Lie groups $G$, such that on $\mathbb R^3$ there is a $G$-invariant foliation with $1$-dimensional connected fibres $C$, but on $\mathbb R^3$ there is no foliation with $2$-dimensional fibres $P$ which consists of a $1$-dimensional foliation of elements of $C$. 
\newline
\noindent
III. Lie groups $G$, such that on $\mathbb R^3$ there is a $G$-invariant foliation with $1$-dimensional connected fibres $C$ and there is on 
$\mathbb R^3$ also a foliation with $2$-dimensional fibres $P$ which consists of a $1$-dimensional foliation of elements of $C$. 
\newline
\noindent
For the groups $G$ belonging to the classes I and II there is a $2$-dimensional manifold $P$ such that $G$ acts on $P$ primitively. According to \cite{gonzalez}, Table 1, p. 341, there is no nilpotent Lie group with this property. Since every subgroup and factor group of a nilpotent Lie group is nilpotent, a nilpotent Lie group $Mult(L)$ is only in the class III. Hence the loop $L$ has a $2$-dimensional connected normal subloop $M$ containing a $1$-dimensional connected normal subloop $N$ of $L$. Since the group $Mult(L)$ is nilpotent the loop $L$ is centrally nilpotent (cf. \cite{bruck}). 
Hence every subloop and every factor loop of $L$ is centrally nilpotent. 
As $L$ is homeomorphic to $\mathbb R^3$ and the multiplication group of the connected centrally nilpotent loop $N$ is a Lie group it is isomorphic to the group $\mathbb R$. The factor loop $L/N$ is a $2$-dimensional connected centrally nilpotent loop.  As $L$ is a fibering of $\mathbb R^3$ over $L/N$ with fibers homeomorphic to $\mathbb R$ it follows that $L/N$ is homeomorphic to $\mathbb R^2$. Every $2$-dimensional connected centrally nilpotent loop which is homeomorphic to $\mathbb R^2$ and having a Lie group as its multiplication group is isomorphic either to the Lie group $\mathbb R^2$ or to an elementary filiform loop (cf. \cite{figula}, Theorem 1). \end{proof}

\begin{Lemma} \label{centrelemma} Let $L$ be a $3$-dimensional proper connected simply connected topological  loop such that its multiplication group $Mult(L)$ is a nilpotent Lie group. The centre $Z$ of the group $Mult(L)$ as well as the centre of the loop $L$ is isomorphic to the group $\mathbb R^n$ with $1 \le n \le 2$. 
\end{Lemma}
\begin{proof} 
By Lemma \ref{simplyconnected} the loop $L$ is homeomorphic to $\mathbb R^3$. 
According to \cite{flugfelder}, p. 25, the centre $Z$ of the group $Mult(L)$ is isomorphic to the centre $Z(L)$ of the loop $L$. 
As the group $Mult(L)$ is nilpotent one has $\hbox{dim} Z \ge 1$ (cf. \cite{handbook}, Proposition 4 (c), p. 619) and $\hbox{dim} Z = \hbox{dim} Z(L)$. As $\hbox{dim} L=3$ we obtain $\hbox{dim} Z \le 3$. For the orbit $Z(e)$ one has $\hbox{dim} Z = \hbox{dim} Z(e)$, otherwise there is a proper 
central subgroup of $Mult(L)$ which leaves the element $e \in L$ fixed  which is a contradiction. 
If $\hbox{dim} Z = 3$, then for the orbit $Z(e)$ one has $Z(e)=L$. Hence $Z$ operates sharply transitively on $L$ and we have 
$Mult(L)= Z \rtimes Inn(L)$. Since we can identify the elements of $L$ with the elements of $Z$ the group $G_l$ topologically generated by the left translations of $L$ contains $Z$ as a subgroup. But Theorem 17.11 in \cite{loops}, p. 231, gives a contradiction and the assertion follows.   
\end{proof}

\begin{Lemma} \label{innermappinglemma} Let $L$ be a $3$-dimensional proper connected simply connected topological loop  such that its multiplication group $Mult(L)$ is a nilpotent Lie group. Let $Z(L)$ be the centre of the loop $L$. 
\newline
\noindent
(a) Every $1$-dimensional normal subloop $N$ of $L$ is a central subgroup of $L$.  
\newline
\noindent
(b) If $\hbox{dim} \ Z(L)=1$ and the factor loop $L/Z(L)$ is isomorphic to the group $\mathbb R^2$ as well as if $\hbox{dim} \ Z(L)=2$, then the inner mapping group $Inn(L)$ of the loop $L$ is abelian. 
\end{Lemma}
\begin{proof} By Lemma \ref{simplyconnected} the loop $L$ is homeomorphic to $\mathbb R^3$. As $Mult(L)$ is nilpotent the proper loop $L$ is centrally nilpotent of class $2$ or $3$.  
Let $N$ be a $1$-dimensional connected normal subloop of $L$. Since the group $Mult(N)$ of $N$ is a Lie subgroup of $Mult(L)$, the loop $N$ is isomorphic to a $1$-dimensional Lie group (cf. \cite{loops}, Proposition 18.18). As $N$ is minimal  it is contained in the centre of $L$ (see \cite{bruck}, Theorem 4 C and Corollary to theorem 4 C, p. 267) and the first assertion is proved. 
\newline
\noindent
In the cases of assertion (b) the loop is centrally nilpotent of class $2$. Hence $L$ has an upper central series 
$e =Z_0 < Z_1 =Z(L) < Z_2=L$, where $Z_i/Z_{i-1}$ is the centre of $L/Z_{i-1}$ for $i=1,2$ such that in the first case $\hbox{dim} Z(L)=1$ and in the second case $\hbox{dim} Z(L) =2$. By Theorem 8 A in \cite{bruck}, pp. 280-281, there exists an ascending series 
$1 =\mathcal{R}_0 < \mathcal{R}_1=\mathcal{R}_2 =Inn(L)$ 
of subgroups of the inner mapping group $Inn(L)$ such that $\mathcal{R}_1/\mathcal{R}_0 \cong Inn(L)$ is an abelian group. This gives the assertion in (b). 
\end{proof}

\begin{Prop} \label{2dimensionalcentre} Let $L$ be a $3$-dimennsional proper connected simply connected topological  loop such that its multiplication group $Mult(L)$ is a nilpotent Lie group. We assume that $L$ has a $2$-dimensional centre  $Z(L)$. Then $Mult(L)$ is a semidirect product of the abelian group $M \cong \mathbb R^m$, $m \ge 3$, by a group $Q \cong \mathbb R$ such that $M= Z \times Inn(L)$, where $\mathbb R^2 =Z \cong Z(L)$ is 
the centre of $Mult(L)$.   
\end{Prop} 
\begin{proof}  By Lemma \ref{simplyconnected} the loop $L$ is homeomorphic to $\mathbb R^3$. The centre $Z(L)$ is a $2$-dimensional normal subgroup of $L$ isomorphic to $\mathbb R^2$ (cf. Lemma \ref{centrelemma}). The factor loop $L/Z(L)$ is a $1$-dimensional connected loop such that its multiplication group is a factor group of $Mult(L)$. Hence $L/Z(L)$ is the group $\mathbb R$ (cf. Theorem 18.18 in \cite{loops}). By Lemma \ref{brucklemma} there exists a normal subgroup $M$ of $Mult(L)$ such that $Mult(L)/M$ is isomorphic to the group $Mult(L/Z(L)) \cong \mathbb R$. The normal subgroup $M$ 
leaves every orbit of $Z(L)$ homeomorphic to $\mathbb R^2$ in the manifold $L$ invariant and contains the multiplication group of $Z(L)$. The multiplication group of $Z(L)$ consists of the translations by elements of $Z(L)$. Hence it is isomorphic to $Z(L)$ and it is also isomorphic to the centre $Z$ of the group $Mult(L)$ (cf. Lemma \ref{centrelemma}). The group $Mult(L)/M$ operates sharply transitively on the orbits of $Z(L)$ in $L$ hence the inner mapping group $Inn(L)$ is a subgroup of $M$. According to Lemma \ref{innermappinglemma} the group $Inn(L)$ is abelian of codimension $3$ in $Mult(L)$. 
Hence one has $M= Z \rtimes Inn(L)$. Therefore, $M$ induces on every orbit $Z(L)(x), x \in L$, the sharply transitive group $\mathbb R^2$. As $Inn(L)$ fixes every element of $Z(L)$ (cf. \cite{bruck}, IV.1) it is a normal subgroup of $M$ such that $Z \cap Inn(L)= \{ 1 \}$. Hence 
$M= Z \times Inn(L)$, where $Z = \mathbb R^2$ is the centre of $Mult(L)$. \end{proof}

\begin{Prop} \label{1dimensionalcentreegy} Let $L$ be a $3$-dimensional proper connected simply connected topological  loop such that its multiplication group $Mult(L)$ is a nilpotent Lie group. We assume that $L$ has a $1$-dimensional centre  $Z(L)$. 
\newline
\noindent
a) If the factor loop $L/Z(L)$ is isomorphic to the 
abelian group $\mathbb R^2$, then $Mult(L)$ is a semidirect product of the abelian group $P \cong \mathbb R^m$, $m \ge 2$, by a group 
$Q \cong \mathbb R^2$ such that $P= Z \times Inn(L)$, where $\mathbb R =Z \cong Z(L)$ is the centre of $Mult(L)$.   
\newline
\noindent
b) If the factor loop $L/Z(L)$ is isomorphic to a $2$-dimensional elementary filiform loop $L_{\mathcal F}$, then there is a normal subgroup $S$ of the group $Mult(L)$  containing the centre $Z$ of $Mult(L)$ such that the factor group $Mult(L)/S$ is an elementary filiform Lie group ${\mathcal F}_{n+2}$ with $n \ge 2$. 
\end{Prop} 
\begin{proof}  By Lemma \ref{simplyconnected} the loop $L$ is homeomorphic to $\mathbb R^3$. The centre $Z(L)$ is a $1$-dimensional connected normal subgroup of $L$ isomorphic to $\mathbb R$ (cf. Lemma \ref{centrelemma}). The factor loop $L/Z(L)$ is a $2$-dimensional connected loop such that its multiplication group is a factor group of $Mult(L)$ (cf. Lemma \ref{brucklemma}). Hence $Mult(L/Z(L))$ is nilpotent and therefore $L/Z(L)$ is isomorphic either to the group $\mathbb R^2$ or to an elementary 
filiform loop $L_{\mathcal F}$ (cf. p. 4). 
\newline
\noindent
In the second case $Mult(L/Z(L))$ is isomorphic to an at least $4$-dimensional elementary filiform Lie group 
${\mathcal F}_{n+2}$, $n \ge 2$ (cf. Theorem 1 in \cite{figula}). Moreover, there exists a normal subgroup $S$ of $Mult(L)$ such that $Mult(L)/S$ is isomorphic to the group $Mult(L/Z(L))$ and the group $S$ contains the Lie group $Mult(Z(L)) \cong \mathbb R$ (cf. Lemma \ref{brucklemma}). The multiplication group of $Z(L)$ consists of the translations by elements of $Z(L)$. Hence it is isomorphic to $Z(L)$ and it is also isomorphic to the centre $Z$ of the group $Mult(L)$ (cf. Lemma \ref{centrelemma}) and the assertion b) is proved.

In the first case $Mult(L/Z(L))$ is isomorphic to $\mathbb R^2$. By Lemma \ref{brucklemma} there exists a normal subgroup $P$ of $Mult(L)$ such that $Mult(L)/P$ is isomorphic to the group $\mathbb R^2$, the group $P$ contains the Lie group $Mult(Z(L)) \cong \mathbb R$ and $P$ leaves every orbit of $Z(L)$ which is homeomorphic to $\mathbb R$ in the manifold $L$ invariant. As $Mult(L)/P \cong \mathbb R^2$ the factor group $Mult(L)/P$ operates sharply transitively on the orbits of $Z(L)$ in $L$. The nilpotent group $P$ induces on the orbit $Z(L)(e)$ the sharply transitive group $\mathbb R$. Therefore $P$ induces on every orbit $Z(L)(x)$, $x \in L$, the sharply transitive group $\mathbb R$. The stabilizer $P_1$ of $e \in L$ in $P$ fixes every point of the orbit $Z(L)(e)=P(e)$. Hence $P_1$ is a normal subgroup of $P$. Since the factor group $P/P_1 \cong \mathbb R$ the commutator subgroup $P'$ of $P$ is contained in $P_1$ and $P'$ is normal in $Mult(L)$. The group $P'$ is trivial, otherwise 
the group $Mult(L)$ would contain the normal subgroup $1 \neq P'$ which has fixed points and 
$Mult(L)$ does not operate effectively on $L$. Hence $P$ is abelian. Since $\hbox{dim} \ Mult(L) \ge 4$ the group $P$ is isomorphic to $\mathbb R^n$, 
$n \ge 2$. The inner mapping group $Inn(L)$ has codimension $3$ and 
hence $Inn(L)$ is the group $P_1$.  The group consisting of the translations by elements of $Z(L)$ is isomorphic to $Z(L)$ and it is isomorphic to the centre $Z$  of $Mult(L)$. Then $P=Z \times Inn(L)$ and the assertion a) is proved. 
\end{proof}

\section{Three-dimensional loops having the four-dimensional filiform group as their left translation group}

Let $h: \mathbb R^2 \to \mathbb R$ be a continuous function and denote by $T_{(a_1,a_2)} h(x,y)=h(x+a_1,y+a_2)$ the translation by $(a_1,a_2) \in \mathbb R^2$. An exponential polynomial on $\mathbb R^2$ is a finite linear combination of terms $x^{q_1} y^{q_2} \exp (\lambda _1 x+ \lambda _2 y)$, where $q_i$ are nonnegative integers and 
$\lambda_i \in \mathbb C$, $i=1,2$. The following proposition follows from \cite{anselone}, Theorem, p. 747, and Section 4, 
p. 751. 

\begin{Prop} \label{exponentialpolynomial} The real vector space $W$ generated by a real continuous function 
$h: \mathbb R^2 \to \mathbb R$ and by the translations $T_{(a_1,a_2)} h(x,y)$, $(a_1,a_2) \in \mathbb R^2$, has finite dimension if and only if it is the span of a set of functions 
\begin{equation} \label{sincosexp} x^{q_1} y^{q_2}(k_1 e^{\lambda _1 x} + c_1 e^{a_1 x} \sin(b_1 x) + d_1 e^{a_1 x} \cos(b_1 x)) \cdot \nonumber \end{equation} 
\begin{equation} (k_2 e^{\lambda _2 y} + c_2 e^{a_2 y} \sin(b_2 y) + d_2 e^{a_2 y} \cos(b_2 y)),  \end{equation}  
where $q_i$ are nonnegative integers, $\lambda _i, a_i, b_i, c_i, d_i, k_i \in \mathbb R$, $i=1,2$. 
\end{Prop} 

\noindent
Every $4$-dimensional simply connected nilpotent Lie group is isomorphic either to the $4$-dimensional filiform Lie group ${\mathcal F}_{4}$ or to the direct product of the $3$-dimensional simply connected non-abelian nilpotent Lie group ${\mathcal F}_{3}$ with the group $\mathbb R$. In this section 
we classify the $3$-dimensional connected simply connected topological  loops having the $4$-dimensional filiform Lie group ${\mathcal F}_4$ as the group topologically generated by their left translations. The multiplication of the loops in this class depends on a continuous real function of one or two variables. The multiplication groups of these loops are Lie groups precisely if the continuous function occuring in the loop multiplication derives from exponential polynomials. 
We prove that the group ${\mathcal F}_4$ cannot be the multiplication group of $3$-dimensional topological loops.  

\begin{Prop} \label{filiform4dim} Let $G$ be the $4$-dimensional filiform Lie group ${\mathcal F}_{4}$ and choose for $G$ the representation on $\mathbb R^4$ by the multiplication 
\[ g(x_1, x_2, x_3, x_4) g(y_1, y_2, y_3, y_4) = \]
\[ g(x_1+y_1, x_2+y_2, x_3+y_3 -x_2 y_1, x_4+y_4- y_1 x_3+ \frac{1}{2} x_2 y_1^2). \] 
Let $H$ be a subgroup of $G$ which is isomorphic to $\mathbb R$ and which is not normal in $G$, then using automorphisms of $G$ we may choose $H$ in one of the following forms: 
\[ H_1=\{ g(0,0,v,0); v \in \mathbb R \}, \ \   H_2=\{ g(v,0,0,0); v \in \mathbb R \}. \]    
\noindent
a) Every continuous sharply transitive section $\sigma : G/H_1 \to G$ with the properties that $\sigma (G/H_1)$ generates $G$ and 
$\sigma (H_1)=1$ is determined by the map $\sigma _f: g(x,y,0,z) H_1 \mapsto g(x, y, f(x,y), z)$, 
where $f: \mathbb R^2 \to \mathbb R$ is a continuous function with $f(0,0)=0$. 
The multiplication of the loop $L_f$ corresponding to $\sigma _f$ can be written as 
\begin{equation} \label{multiplication4dimelso} (x_1,y_1,z_1) \ast (x_2,y_2,z_2) = \nonumber \end{equation} 
\begin{equation} (x_1+x_2, y_1+y_2, z_1+z_2- x_2 f(x_1,y_1)+ \frac{1}{2} x_2^2 y_1). \end{equation} 
\noindent
The multiplication group $Mult(L_f)$ of the loop $L_f$ is a Lie group precisely if $f(x,y)$ is a finite linear combination of 
functions given by (\ref{sincosexp}). Moreover, if $f(x,y)=f(y)$ has the form $f(y)=\sum \limits _{i=1}^n b_i y^i$, $b_i \in \mathbb R$, then the group $Mult(L_f)$ is isomorphic to the direct product ${\mathcal F}_4 \times _Z {\mathcal F}_{n+2}$ of the elementary filiform Lie groups ${\mathcal F}_4$ and ${\mathcal F}_{n+2}$ with amalgamated centre $Z$, $n \ge 1$. 
\newline
\noindent
b) Each continuous sharply transitive section $\sigma : G/H_2 \to G$ such that $\sigma (G/H_2)$ generates $G$ and 
$\sigma (H_2)=1$ is determined by the map 
\[ \sigma _h: g(0,x,y,z) H_2 \mapsto g(h(x), x, y- x h(x), z - y h(x) +\frac{1}{2} x h(x)^2),  \]
where $h: \mathbb R \to \mathbb R$ is a non-linear continuous function with $h(0)=0$. 
The multiplication of the loop $L_h$ corresponding to $\sigma _h$ can be written as 
\begin{equation} \label{multiplication4dimmasodik} (x_1,y_1,z_1) \ast (x_2,y_2,z_2) = \nonumber \end{equation} 
\begin{equation} (x_1+x_2, y_1+y_2+ x_2 h(x_1), z_1+z_2+ y_2 h(x_1)+ \frac{1}{2} x_2 h(x_1)^2). \end{equation} 
\newline
\noindent
The multiplication group $Mult(L_v)$ of the loop $L_v$ is a Lie group if and only if the function $h(x)$ is a finite linear combination of functions 
\begin{equation} \label{loopv} x^{q}(c_1 e^{\lambda  x}+ c_2 e^{a x} \sin(b x) + c_3 e^{a x} \cos(b x)), \nonumber  \end{equation}
where $q$ is a nonnegative integer, $\lambda , a, b, c_1, c_2, c_3 \in \mathbb R$. 
\newline
\noindent 
c) The group $G$ cannot be the multiplication group of a topological loop homeomorphic to $\mathbb R^3$.  
\end{Prop} 
\begin{proof} The Lie algebra ${\bf g}$ of $G$ is given by the basis $\{ e_1, e_2, e_3, e_4 \}$ with $[e_1, e_2]=e_3$, $[e_1, e_3]=e_4$. 
We can represent the elements of $G$ as the matrices 
\begin{equation}  g(x_1,x_2,x_3,x_4)= \left( \begin{array}{cccc} 
1 & x_3 & x_2  & x_4 \\
0 & 1 & 0 & -x_1  \\
0 & -x_1 & 1 & \frac{x_1^2}{2} \\
0 & 0 & 0 & 1 \end{array} \right)  \nonumber \end{equation}
with $x_i \in \mathbb R, i=1,2,3,4$. 
Hence the multiplication of the group $G$ can be represented on $\mathbb R^4$ as given in the assertion. 
The subgroup $\exp t e_4$, $t \in \mathbb R$, is the centre of $G$, the subgroup $\exp (t e_3 + s e_4)$, 
$t, s \in \mathbb R$, is the commutator subgroup of $G$.  Hence the automorphism group 
of ${\bf g}$ consists of linear mappings 
\[ \varphi(e_1)= a_1 e_1 + a_2 e_2 + a_3 e_3+ a_4 e_4, \ \ \varphi(e_2)= b_1 e_1+ b_2 e_2 + b_3 e_3 + b_4 e_4, \ \ \]
\[ \varphi(e_3)= (a_1 b_2- a_2 b_1) e_3+ (a_1 b_3- a_3 b_1) e_4, \ \ \varphi(e_4)= a_1 (a_1 b_2-a_2 b_1) e_4, \]
with $a_1 b_2- a_2 b_1 \neq 0, a_1 \neq 0$, $a_2, a_3, a_4, b_2, b_3, b_4 \in \mathbb R$.  
\newline
\noindent   
Let $H$ be a subgroup of $G$ which is isomorphic to $\mathbb R$ and which is not normal in $G$. Then $H$ is a subgroup 
$\exp t(\alpha e_1 + \beta e_2 + \gamma e_3 + \delta e_4)$ with $t \in \mathbb R$ and 
$\alpha ^2 + \beta ^2 + \gamma ^2=1$. Then a suitable automorphism of $G$ corresponding to an automorphism $\varphi $ of 
${\bf g}$ maps $H$ onto one of the following subgroups 
\[ H_1= \exp t e_3,  \ H_2=\exp t e_1. \] 
\noindent
First we assume that $H=H_1=\{ g(0,0,k,0); \ k \in \mathbb R \}$. Since all elements of $G$ have a unique decomposition as 
$g(x, y, 0, z)  g(0, 0, k, 0)$, any continuous function
$f: \mathbb R^3 \to \mathbb R; (x,y, z) \mapsto f(x,y,z)$ determines a continuous section $\sigma : G/H \to G$ given by 
\begin{equation} \label{section4dimelso} g(x,y,0,z) H \mapsto  g( x, y, 0, z) g(0, 0, f(x,y,z), 0) = g(x, y, f(x,y,z), z). \nonumber \end{equation}  
The section $\sigma $ is sharply transitive if and only if for every triples $(x_1,y_1,z_1)$, $(x_2,y_2,z_2) \in \mathbb R^3$ there exists precisely one triple  $(x,y,z) \in \mathbb R^3$ such that 
\[ g(x, y, f(x,y,z), z) g(x_1, y_1, 0, z_1)= g(x_2, y_2, 0, z_2) g(0,0,t,0) \] 
for a suitable $t \in \mathbb R$. This gives the equations
\[x= x_2-x_1, \  y= y_2-y_1, \ t= f(x_2-x_1,y_2-y_1,z) - (y_2-y_1) x_1, \]
\begin{equation} \label{equ5} 0= z +z_1 -z_2 + \frac{1}{2} (y_2 -y_1) x_1^2 - x_1 f(x_2-x_1,y_2 -y_1,z). \nonumber \end{equation}
These are equivalent to the condition that for every $x_0= x_2-x_1$, $y_0= y_2- y_1$ and $x_1 \in \mathbb R$ the function 
$g: z \mapsto z - x_1 f(x_0, y_0, z): \mathbb R \to \mathbb R$ is a bijective mapping. Let be $\psi_1 < \psi_2 \in \mathbb R$ then 
$g(\psi_1) \neq g(\psi_2)$, e.g. $g(\psi _1) < g(\psi _2)$. We consider 
\[ 0 <  g(\psi _2) - g(\psi _1)= \psi _2 - \psi _1 - x_1 [f(x_0, y_0, \psi_2 ) - f(x_0, y_0, \psi_1 )] \]
as a linear function of $x_1 \in \mathbb R$. If $f(x_0, y_0, \psi_2 ) \neq f(x_0, y_0, \psi_1 )$, then there exists a $x_1 \in \mathbb R$ 
such that $g(\psi _2) - g(\psi _1) =0$, which is a contradiction. Hence the function $f(x, y, z)=f(x,y)$ does not depend on $z$. In this case $g$ is a monotone function and every continuous function $f(x,y)$ with $f(0,0)=0$ determines a loop multiplication. 
\newline
\noindent
We represent the loop $L_f$ in the coordinate system $(x,y,z) \mapsto g(x,y,0,z)H$.  The multiplication  
$(x_1,y_1,z_1) \ast (x_2,y_2,z_2)$ is determined if we apply $g(x_1,y_1,0,z_1)H$$=g(x_1, y_1, f(x_1,y_1), z_1)$ 
to the left coset $g(x_2,y_2,0,z_2)H$ and find in the image coset the element of $G$ which is in the set 
$\{ g(x,y,0,z)H; x, y, z \in \mathbb R \}$. A direct calculation 
gives the multiplication (\ref{multiplication4dimelso}) in the assertion. 
\newline 
\noindent 
This loop is proper precisely if the set $\sigma (G/H)=\{ g(x, y, f(x,y), z); \ x,y,z \in \mathbb R \}$ 
generates the whole group $G$.    
The set $\sigma (G/H)$ contains the subset 
\[ F=\{ g(x, y, f(x,y), 0); \ x,y \in \mathbb R \} \] and 
the centre  $Z=\{ g(0,0,0,z); \ z \in \mathbb R \}$ of $G$. The set $\sigma (G/H)$ generates $G$ if and only if the projection of the group $\langle F \rangle $ generated by the set $F$ onto the set $S=\{ g(k,l,m,0); k,l,m \in \mathbb R \}$ has dimension $3$. The set $F$ contains the subsets 
$F_1=\{ g(x, 0, f(x,0), 0); \ x \in \mathbb R \}$ and $F_2=\{ g(0, y, f(0,y), 0); \ y \in \mathbb R \}$. Therefore $\sigma (G/H)$ generates $G$ 
if the projection of the group $\langle F_1 \rangle $ generated by the set $F_1$ onto the set $S$ has dimension $2$.
This is the case if the group  $\langle F_1 \rangle $ is not a $1$-parameter subgroup. But $F_1$ is a $1$-parameter subgroup if and only if $f(x,0)= \lambda x, \lambda \in \mathbb R$. In this case the set $\sigma (G/H)$ generates $G$ if there exists an element $h \in F_2$ such that the projection of the set $h F_1 h^{-1}$ onto the set $S$ is different from 
$F_1$. Since for every $h=g(0, y, f(0,y), 0)$ with $y \neq 0$ we have $pr(g(0, y, f(0,y), 0) g(x, 0, \lambda x, 0) 
g(0, -y, -f(0,y), 0))= g(x,0,\lambda x -x y,0) \notin F_1$ the set $\sigma (G/H)$ generates $G$ for arbitrary continuous function $f(x,y)$ with $f(0,0)=0$.  

\noindent
The right translation $\rho _{(a,b,c)}$ of the loop $L_f$ is the map 
\begin{equation} \label{right1equ} \rho _{(a,b,c)}: (x,y,z) \mapsto (x,y,z) \ast (a,b,c)= (x+a, y+b, z+c-a f(x,y)+ \frac{1}{2} a^2 y). \nonumber \end{equation} 
Its inverse map $\rho _{(a,b,c)}^{-1}$ is given by 
\begin{equation} \label{right2equ} \rho _{(a,b,c)}^{-1}: (x,y,z) \mapsto (x-a, y-b, z-c+a f(x-a,y-b)- \frac{1}{2} a^2 (y-b)). \nonumber \end{equation} 
Since $\rho_{(0,d_1,e_1)} \rho_{(0,d_2,e_2)}= \rho_{(0,d_1+d_2, e_1+e_2)}$ and  
one has \[ \rho _{(a,b,c)} \rho_{(0,d,e)} \rho _{(a,b,c)}^{-1} = \rho _{(0,d,e+ a (f(x-a,y-b)- f(x-a,y-b+d)) + \frac{1}{2} a^2 d)} \]
the group $G_{\rho }$ topologically generated by the right translations of $L_f$ has a normal subgroup 
$N_{\rho }= \{ \rho _{(0,d,e)}; d,e \in \mathbb R \} \cong \mathbb R^2$. Because of 
 $\rho _{(a,b,c)} = \rho _{(0,b,c)} \rho _{(a,0,0)}$ one has $G_{\rho }= N_{\rho } \Sigma $, 
where $\Sigma $ is the group generated by the set $\{ \rho _{(a,0,0)}; a \in \mathbb R \}$. 
The group $G_{\rho }$ and hence the multiplication group $Mult(L_f)$ is a finite dimensional Lie group precisely if the group 
$\Sigma $ has finite dimension. 
As $\rho _{(a_2,0,0)} \rho _{(a_1,0,0)}$ is the map $(x,y,z) \mapsto (x+a_1+a_2,y,z-a_1 f(x,y)-a_2 f(x+a_1,y)+ \frac{1}{2} (a_1^2+a_2^2)y)$ the group $\Sigma $ is a subgroup of the transformation group $\Gamma $ consisting of the maps 
\begin{equation} \label{sigmaelso} 
\gamma _{(t_1, t_2, \alpha _i, \beta _i)}: (x,y,z) \mapsto (x+t_1,y,z+ \sum \alpha _i f(x+ \beta _i,y)+ \frac{1}{2} t_2 y), 
\nonumber \end{equation} 
where $t_1, t_2, \alpha _i, \beta _i \in \mathbb R$. By Proposition \ref{exponentialpolynomial} the group 
$\Gamma $ and hence the group $\Sigma $ is a finite dimensional Lie group precisely if the function $f(x,y)$ is a finite linear combination of functions given by (\ref{sincosexp}). 
\newline
\noindent
The group $G_{\rho }$ contains the $1$-parameter subgroups 
$S_1=\{ \rho _{(0,t,0)}: (x,y,z) \mapsto (x, y+t, z); t \in \mathbb R \}$ and 
$S_2=\{ \rho _{(0,0,t)}: (x,y,z) \mapsto (x, y, z+t); t \in \mathbb R \}$. Moreover, if $f(x,y)=f(y)$, then $G_{\rho }$  contains also the $1$-parameter subgroups 
\begin{equation} \label{righttrans} 
S_3^a=\{ \rho _{(a,0,0)}^t: (x,y,z) \mapsto (x+ t a, y, z-t a f(y)+ \frac{1}{2} t a^2 y); t \in \mathbb R \} \nonumber \end{equation}
for all $a \in \mathbb R$.  
The tangent vectors of the $1$-parameter subgroups $S_i$, $i=1,2,3$, at $1 \in G_{\rho }$ are 
$X_1=\frac{\partial }{\partial y}$, $X_2=\frac{\partial }{\partial z}$, 
$X_3^a=\frac{\partial }{\partial x} - f(y) \frac{\partial }{\partial z}+ \frac{1}{2} a y \frac{\partial }{\partial z}$ for all $a \in \mathbb R \setminus \{ 0 \}$. Hence the Lie algebra ${\bf g}_{\rho }$ of the group $G_{\rho }$ is given by 
${\bf g}_{\rho }= \langle \frac{\partial }{\partial y}, \frac{\partial }{\partial z}, y \frac{\partial }{\partial z}, \frac{\partial }{\partial x} - f(y) \frac{\partial }{\partial z} \rangle $. 
Already in the case that the function $f(x,y)$ does not depend on the variable $x$ the group 
$G_{\rho }$ is different from the group ${\mathcal F}_4$ and hence for any function $f(x,y)$ the multiplication group $Mult(L_f)$ of the loop $L_f$ cannot be isomorphic to ${\mathcal F}_4$.   
\newline
\noindent
The Lie algebra ${\bf g}$ of the group $G={\mathcal F}_{4}$ topologically generated by all left translations of the loop 
$L_f$ is given by 
${\bf g}=\langle \frac{\partial }{\partial x}, x \frac{\partial }{\partial z}, \frac{\partial }{\partial z}, 
x^2 \frac{\partial }{\partial z} \rangle $. 
\newline
\noindent 
If $f(x,y)= f(y)= \sum \limits _{j=1}^n b_n y^j$, $b_n \in \mathbb R$, then the tangent space of the manifold 
$\mathcal R =\{ \rho _x; x \in L_f \}$ at the identity of $G_{\rho }$ generates the Lie algebra 
${\bf g_{\rho }}= \{ \frac{\partial }{\partial y}, \frac{\partial }{\partial z}, \frac{\partial }{\partial x}, 
y \frac{\partial }{\partial z} \}$ if $f(y)=0$  and  
${\bf g_{\rho }}= \{ \frac{\partial }{\partial y}, \frac{\partial }{\partial z}, \frac{\partial }{\partial x} - b_n y^n \frac{\partial }{\partial z}, y^{n-1} \frac{\partial }{\partial z}, \cdots , y \frac{\partial }{\partial z} \}$ if 
$b_n \neq 0$. Hence the set of the basis elements of the Lie algebra of 
$Mult(L_f)$ is $\{ \frac{\partial }{\partial x}, x^2 \frac{\partial }{\partial z}, 
x \frac{\partial }{\partial z}, \frac{\partial }{\partial y}, y^n \frac{\partial }{\partial z}, 
y^{n-1} \frac{\partial }{\partial z}, \cdots , y \frac{\partial }{\partial z}, \frac{\partial }{\partial z} \}$, $n \ge 1$. Therefore the multiplication group $Mult(L_f)$ is the direct product ${\mathcal F}_{4} \times _Z {\mathcal F}_{n+2}$ of the elementary filiform Lie groups ${\mathcal F}_{4}$ and ${\mathcal F}_{n+2}$, $n \ge 1$, with amalgamated centre $Z$. Hence the assertion a) is proved. 

\medskip
\noindent 
Now we consider the case that $H=H_2=\{ g(v, 0, 0, 0); \ v \in \mathbb R \}$. Since all elements of $G$ have a unique decomposition as 
$g(0, x, y, z)  g(v, 0, 0, 0)$, any continuous function
$h: \mathbb R^3 \to \mathbb R; (x,y, z) \mapsto h(x,y,z)$ determines a continuous section $\sigma :G/H \to G$ given by
\[ \sigma: g(0,x,y,z) H \mapsto  g(0, x, y, z) g(h(x,y,z), 0, 0, 0)= \]
\[ g(h(x,y,z), x, y- x h(x,y,z), z - y h(x,y,z) + \frac{1}{2} x h(x,y,z)^2).  \]
The section $\sigma $ is sharply transitive if and only if for every triples $(x_1,y_1,z_1)$, $(x_2,y_2,z_2) \in \mathbb R^3$ there exists precisely one triple  $(x,y,z) \in \mathbb R^3$ such that
\begin{equation}  g(h(x,y,z), x, y- x h(x,y,z), z - y h(x,y,z) + \frac{1}{2} x h(x,y,z)^2)g(0,x_1,y_1,z_1)=  \nonumber \end{equation}  
\begin{equation} g(0,x_2,y_2,z_2) g(t,0,0,0) \nonumber \end{equation}
for a suitable $t \in \mathbb R$. This gives the equations $x=x_2-x_1$, $t=h(x_2-x_1,y,z)$, 
\begin{equation} \label{equ2} y + x_1 h(x_2-x_1, y, z)= y_2-y_1,  \end{equation} 
\begin{equation} \label{equ1} z +  y_1 h(x_2-x_1,y,z) + \frac{1}{2} x_1 h(x_2-x_1,y,z)^2= z_2-z_1. \end{equation}
\newline
For $x_1=0$ equation (\ref{equ2}) yields that $y=y_2-y_1$ and equation (\ref{equ1}) reduces to 
\begin{equation} \label{equuj1} z  = z_2-z_1 -y_1 h(x_2,y_2-y_1,z). \nonumber \end{equation}
This equation has a unique solution for $z$ if and only if 
for every $x_0= x_2$, $y_0= y_2- y_1$ and $y_1 \in \mathbb R$ the function 
$f: z \mapsto z + y_1 h(x_0, y_0, z): \mathbb R \to \mathbb R$ is a bijective mapping. Let be $\psi_1 < \psi_2 \in \mathbb R$ then 
$f(\psi_1) \neq f(\psi_2)$, e.g. $f(\psi _1) < f(\psi _2)$. We consider 
\[ 0 <  f(\psi _2) - f(\psi _1)= \psi _2 - \psi _1 + y_1 [h(x_0, y_0, \psi_2 ) - h(x_0, y_0, \psi_1 )] \]
as a linear function of $y_1 \in \mathbb R$. If $h(x_0, y_0, \psi_2 ) \neq h(x_0, y_0, \psi_1 )$, then there exists a $y_1 \in \mathbb R$ 
such that $f(\psi _2) - f(\psi _1) =0$, which is a contradiction. Hence the function $h(x, y, z)=h(x,y)$ does not depend on $z$. 
Using this, equations (\ref{equ2}) and (\ref{equ1}) give  
\begin{equation} \label{equ2masodik} y = y_2-y_1 -x_1 h(x_2-x_1, y),  \end{equation} 
\begin{equation} \label{equ1masodik} z = z_2-z_1 -y_1 h(x_2-x_1,y) - \frac{1}{2} x_1 h(x_2-x_1,y)^2. \end{equation}
\noindent
Equation (\ref{equ2masodik}) has a unique solution for $y$ precisely if for every $x_0= x_2-x_1$ and $x_1 \in \mathbb R$ the function 
$f: y \mapsto y + x_1 h(x_0, y): \mathbb R \to \mathbb R$ is a bijective mapping. Let be $\psi_1 < \psi_2 \in \mathbb R$ then 
$f(\psi_1) \neq f(\psi_2)$, e.g. $f(\psi _1) < f(\psi _2)$. We consider 
$0 <  f(\psi _2) - f(\psi _1)= \psi _2 - \psi _1 + x_1 [h(x_0, \psi_2 ) - h(x_0, \psi_1 )]$ 
as a linear function of $x_1 \in \mathbb R$. If $h(x_0, \psi_2 ) \neq h(x_0, \psi_1 )$, then there exists a 
$x_1 \in \mathbb R \setminus \{ 0 \}$ 
such that $f(\psi _2) - f(\psi _1) =0$, which is a contradiction. Hence the function $h(x, y)=h(x)$ does not depend on $y$. But then the unique solution of (\ref{equ1masodik}) is $z=z_2-z_1 -y_1 h(x_2-x_1) - \frac{1}{2} x_1 h(x_2-x_1)^2$. Hence each continuous function $h(x)$ with $h(0)=0$ determines a loop multiplication. 
\newline 
\noindent 
The multiplication of the loop $L_h$ can be expressed in the coordinate system
$(x,y,z) \mapsto g(0,x,y,z)H$ if we apply $\sigma _h(g(0,x_1,y_1,z_1)H)= $
\begin{equation} =g(h(x_1),x_1,y_1- x_1 h(x_1), z_1 - y_1 h(x_1) +\frac{1}{2} x_1 h(x_1)^2) \nonumber \end{equation}
to the left coset $g(0,x_2,y_2,z_2)H$ and find in the image coset the element of $G$ which lies in the set 
$\{g(0,x,y,z)H; \ x,y,z \in \mathbb R \}$.
A direct computation yields the multiplication (\ref{multiplication4dimmasodik}) in the assertion. 
\newline
\noindent
This loop is proper precisely if the set 
\[ \sigma (G/H)=\{ g(h(x),x,y- x h(x), z - y h(x) +\frac{1}{2} x h(x)^2); \\ x,y,z \in \mathbb R \}  \]
generates the whole group $G$.    
The set $\sigma (G/H)$ contains the commutator subgroup $G'=\{ g(0,0,y,z); \ y,z \in \mathbb R \}$ of $G$ and the subset 
\[ F=\{ g(h(x),x, - x h(x), \frac{1}{2} x h(x)^2);  \ x \in \mathbb R \}. \]  
As 
$g(h(x_1),x_1, - x_1 h(x_1), \frac{1}{2} x_1 h(x_1)^2) g(h(x_2),x_2, - x_2 h(x_2), \frac{1}{2} x_2 h(x_2)^2) = $ 
\begin{equation} g(h(x_1)+h(x_2), x_1+x_2, -x_1 h(x_1) - (x_2+x_1) h(x_2), \nonumber \end{equation} 
\begin{equation} \frac{1}{2} (x_1+x_2) h(x_2)^2 + \frac{1}{2} x_1 h(x_1)^2 + x_1 h(x_1) h(x_2)),  \nonumber \end{equation} 
the group $G'$ and the subgroup $\langle F \rangle $ topologically generated by the set $F$ generate $G$ precisely if the mapping assigning to the second component of any element of $\langle F \rangle $
its first component is not a homomorphism. This occurs if and only if the function $h$ is non-linear. 

\noindent
The right translation $\rho _{(a,b,c)}$ of the loop $L_h$ is the map 
\begin{equation} \label{right1equuj} (x,y,z) \mapsto (x,y,z) \ast (a,b,c)= (x+a, y+b+a h(x), z+c+ b h(x)+ \frac{1}{2} a h(x)^2). \nonumber \end{equation} 
Its inverse map $\rho _{(a,b,c)}^{-1}$ is given by 
\begin{equation} \label{right2equuj} (x,y,z) \mapsto (x-a, y-b-a h(x-a), z-c-b h(x-a)- \frac{1}{2} a h(x-a)^2). \nonumber \end{equation} 
As $\rho _{(0,d_2,e_2)} \rho _{(0,d_1,e_1)} = \rho _{(0,d_1+d_2,e_1+e_2)}$ and 
$\rho _{(a,b,c)}^{-1} \rho_{(0,d,e)} \rho _{(a,b,c)}= \rho _{(0,d,e+ d (h(x+a)- h(x)))}$ 
 the group $G_{\rho }$ topologically generated by all right translations of $L_h$ contains an abelian normal subgroup 
 $N_{\rho }= \{ \rho _{(0,d,e)}; d,e \in \mathbb R \}$. Because of 
 $\rho _{(a,b,c)} = \rho _{(0,b,c)} \rho _{(a,0,0)}$ one has $G_{\rho }= N_{\rho } \Sigma $, 
where $\Sigma $ is the group generated by the set $\{ \rho _{(a,0,0)}; a \in \mathbb R \}$. The group $G_{\rho }$ and hence the multiplication group $Mult(L_h)$ is a finite dimensional Lie group precisely if the group $\Sigma $ has finite dimension. 
Since  $\rho _{(a_2,0,0)} \rho _{(a_1,0,0)}$ is the map $(x,y,z) \mapsto (x+a_1+a_2,y+a_1 h(x)+a_2 h(x+a_1),z+ \frac{1}{2} a_1 h(x)^2+ \frac{1}{2} a_2 h(x+a_1)^2)$ the group $\Sigma $ is a subgroup of the transformation group $\Gamma $ consisting of the maps 
\begin{equation} \label{righth5equ} \gamma _{t,\alpha _i, \beta _i}:(x,y,z) \mapsto (x+t, y+ \sum \alpha _i h(x+ \beta _i), 
z+ \frac{1}{2} \sum \alpha _i h(x+ \beta _i)^2), \nonumber \end{equation}
where $t, \alpha _i, \beta _i \in \mathbb R$. It follows from Theorem in \cite{anselone}, p. 747, that the group $\Gamma $ and hence $\Sigma $ is a finite dimensional Lie group precisely if the function $h(x)$ has the form as in the assertion b). Hence the assertion b) is proved.

\medskip
\noindent
If the group $G$ is also the group topologically generated by all left and right translations of a topological loop $L_h$ defined in case b), then the inner mapping group of $L_h$ is the subgroup $Inn(L_h)=H_2=\{ g(v, 0, 0, 0); \ v \in \mathbb R \}$. Since the loop $L_h$ satisfies the conditions of Proposition \ref{1dimensionalcentreegy} a) the group $G=Mult(L_h)=P \rtimes Q$, such that $P=Z \times Inn(L_h)$ is a $2$-dimensional abelian normal subgroup of $G$ containing the centre $Z$ of $G$. The only abelian normal subgroup with the desired properties is $P=\{ g(0,0,x_3,x_4); x_3, x_4 \in \mathbb R \}$. But 
$H_2$ is not contained in $P$. This contradiction shows that $G$ is not the multiplication group of a loop $L_h$. Hence the assertion c) is proved. \end{proof}

\section{$3$-dimensional topological loops with $2$-dimensional centre}

Let $L$ be a proper connected simply connected topological loop  which has a $2$-dimensional centre. 
By Proposition \ref{2dimensionalcentre} the multiplication group $Mult(L)$ has a $2$-dimensional centre $\cong \mathbb R^2$ and $Mult(L)$ contains an abelian normal subgroup $M$ of codimension $1$. If $\hbox{dim} \ Mult(L) \le 5$, then according to the classification of nilpotent Lie algebras in \cite{handbook}, pp. 650-652, the group  $Mult(L)$ is isomorphic either to the direct products ${\mathcal F}_{3} \times \mathbb R$ or to ${\mathcal F}_{4} \times \mathbb R$, 
where ${\mathcal F}_{n+2}$ is the elementary filiform Lie group of dimension $n+2$, or it is the unique $5$-dimensional indecomposable simply connected nilpotent Lie group such that its $2$-dimensional centre coincides with the commutator subgroup. 

First we classify all $3$-dimensional simply connected topological loops $L$ having the group ${\mathcal F}_{3} \times \mathbb R$ as the group topologically generated by the left translations of $L$. The multiplication of the loops  in this class depends on a continuous real function $v(x,z)$ of two variables. The multiplication group of the loops $L_v$ is a Lie group precisely if the function $v(x,z)$ derives from an exponential polynomial. 
The function $v(x,z)=v(x)$ is a polynomial of degree $n$ if and only if the multiplication group of $L_v$ is isomorphic to the group ${\mathcal F}_{n+2} \times \mathbb R$, $n >1$. Moreover, we give examples for $3$-dimensional topological loops $L$ such that the group topologically generated by all left translations is isomorphic to the group ${\mathcal F}_{m+2} \times \mathbb R$ with $m >1$ and coincides with the multiplication group of $L$.  

\begin{Prop} \label{directproduct} 
(i) Every $3$-dimensional proper connected simply connected topological  loop $L$ having the group ${\mathcal F}_{3} \times \mathbb R$ as the group topologically generated by the left translations of $L$ is given by the multiplication 
{\begin{equation} \label{multiplileft2} (x_1,y_1,z_1) \ast (x_2,y_2,z_2)= (x_1+x_2, y_1+y_2-x_2 v(x_1,z_1), z_1+z_2),  \end{equation} }
\noindent
where $v: \mathbb R^2 \to \mathbb R$ is a continuous function with $v(0,0)=0$ such that 
the function $v$ does not fulfill the identities $v(x,0)=c x$ and $v(0,z)= d z$, $c,d \in \mathbb R$, simultaneously. 
\noindent 
The multiplication group $Mult(L_{v})$ of the loop $L_{v}$ is a Lie group precisely if $v(x,z)$ is a finite linear combination 
of functions given by (\ref{sincosexp}). Moreover, the group $Mult(L_v)$ is isomorphic to ${\mathcal F}_{n+2} \times \mathbb R$ with $n \ge 2$ if and only if the continuous function $v:\mathbb R^2 \to \mathbb R$ depends only on the variable $x$ and $v(x)$ is a polynomial of degree $n$.  
\newline
\noindent 
(ii) Let $L$ be a connected simply connected topological loop $L=(\mathbb R^3, \ast )$ with the multiplication 
\begin{equation} \label{nilpotentszorzas} (x_1,y_1,z_1) \ast (x_2,y_2,z_2)= \nonumber \end{equation} 
\begin{equation} (x_1+x_2, y_1+y_2-x_2 v_1(x_1,z_1)+ \cdots + \frac{(-1)^n}{n!} x_2^n  v_n(x_1,z_1), z_1+z_2),  \end{equation} 
\noindent
where $v_i: \mathbb R^2 \to \mathbb R$, $i=1, 2, \cdots ,n,$  are continuous functions with $v_i(0,0)=0$ such that 
the function $v_n$ does not fulfill the identities $v_n(x,0)=c x$ and $v_n(0,z)= d z$, $c,d \in \mathbb R$, simultaneously. 
Then the group $G$ topologically generated by the left translations of $L_{v_1, \cdots ,v_n}$ is isomorphic to the group ${\mathcal F}_{n+2} \times \mathbb R$, $n \ge 1$.  
\noindent 
For $n >1$ the group $G$ coincides with the multiplication group $Mult(L_{v_1, \cdots ,v_n})$ 
of $L_{v_1, \cdots ,v_n}$ if and only if all functions $v_i$, $i=1,2, \cdots ,n$, depend only on the variable $x$ and there are continuous functions 
$s_i: \mathbb R \to \mathbb R$ with $s_i(0)=0$, $i=1, \cdots ,n$,  such that for all $x,k \in \mathbb R$ the equation 
\begin{equation} \label{equtwoside} -x (s_1(k)+v_1(k))+ \frac{x^2}{2!} (s_2(k)+ v_2(k))+ \cdots +(-1)^n \frac{x^n}{n!} (s_n(k) + v_n(k)) \nonumber \end{equation}
\begin{equation} =-k v_1(x)+ \frac{k^2}{2!} v_2(x)+ \cdots +(-1)^n \frac{k^n}{n!} v_n(x) \end{equation} 
\noindent
holds. 
\newline
\noindent
(iii) Every loop $L_{v_1, \cdots ,v_n}$ defined by (\ref{nilpotentszorzas}) such that the multiplication group of 
$L_{v_1, \cdots ,v_n}$ coincides with the group topologically generated by the left translations of $L_{v_1, \cdots ,v_n}$ 
 is the direct product of the group $\mathbb R$ and a topological loop 
$(M, \circ )$ given by the multiplication 
\begin{equation} \label{nilpotentszorzasujuj} (x_1,y_1) \circ (x_2,y_2)= (x_1+x_2, y_1+y_2-x_2 v_1(x_1)+ \cdots + \frac{(-1)^n}{n!} x_2^n  v_n(x_1)).   \nonumber \end{equation} 
\noindent
The loop $M$ has the elementary filiform Lie group ${\mathcal F}_{n+2}$, $n \ge 2$, as the multiplication group (cf. Theorem  2, \cite{figula}, p. 420). 
\end{Prop} 
\begin{proof} 
The  Lie algebra ${\bf g}=\langle e_1, e_2, \cdots , e_{n+2}, e_{n+3} \rangle $, $n \ge 1$, of the Lie group 
$G={\mathcal F}_{n+2} \times \mathbb R$ is isomorphic to the Lie algebra of all matrices of the form 
{\small \begin{equation} \left \{ \begin{pmatrix} 
0 & a_1 & a_2 & \dots & a_{n-1} & a_n & b & d\\
0 & 0 & 0 &  \dots & 0 & 0 & -c & 0\\
0 & -c & 0 & \dots & 0 & 0 & 0 & 0 \\
0 & 0  & -c & \dots & 0& 0 & 0 & 0 \\
\hdotsfor{8} \\
0 & 0  & 0 & \dots & -c & 0 & 0 & 0 \\
0 & 0  & 0 & \dots & 0 & 0 & 0 & 0 \\
0 & 0  & 0 & \dots & 0 & 0 & 0 & 0 \end{pmatrix}; \ a_i,b,c,d \in \mathbb R, i=1,2,\cdots , n  \right \}. \nonumber \end{equation} }
\noindent
Hence we can represent the elements of the Lie group $G$ as the matrices 
{\small \begin{equation}  g(c,a_1,a_2, \cdots , a_{n-1}, a_n, b, d)= \nonumber \end{equation}
\begin{equation} \label{pmatrixujuj} \begin{pmatrix} 
1 & a_1 & a_2  & \dots & a_{n-1} & a_n & b & d \\
0 & 1 & 0  & 0 & \dots  & 0 & -c & 0 \\
0 & -c & 1 & 0 & \dots  & 0 & \frac{c^2}{2!} & 0 \\
0 & \frac{c^2}{2!}  & -c & 1 & \dots & 0 & - \frac{c^3}{3!} & 0 \\
\hdotsfor{8} \\
0 & \frac{(-1)^{n-1}}{(n-1)!} c^{n-1}  & \frac{(-1)^{n-2}}{(n-2)!}  c^{n-2} & \dots  & \frac{(-1)}{(n-(n-1))!} c^{1}  & 1 & \frac{(-1)^n}{n!} c^n & 0 \\
0 & 0  & 0 & \dots  &  0 & 0 & 1 & 0 \\
0 & 0  & 0 & \dots  &  0 & 0 & 0 & 1 \end{pmatrix},  \end{equation} }
with $a_i, b, c, d \in \mathbb R, i=1,2, \cdots , n$. 
\newline
\noindent
First we treat the group $G={\mathcal F}_{3} \times \mathbb R$. 
The Lie algebra ${\bf g}$ of $G$ is given by the basis 
$\{ e_1, e_2, e_3, e_4 \}$ with $[e_1, e_2]= e_3$ and $[e_4, e_i]=0$, $(i=1,2,3)$. 
The subgroup $\exp t e_3$, $t \in \mathbb R$, is the commutator subgroup of $G$, the subgroup 
$\exp (t e_3+ s e_4)$, $t, s \in \mathbb R$, is the centre of $G$. 
Hence each  automorphism $\varphi $  of ${\bf g}$ has the form 
\[ \varphi(e_1)= a_1 e_1 + a_2 e_2 + a_3 e_3 + a_4 e_4, \ \ \varphi(e_2)= b_1 e_1+ b_2 e_2 + b_3 e_3 + b_4 e_4, \ \ \]
\[ \varphi(e_3)= (a_1 b_2-a_2 b_1) e_3, \ \ \varphi(e_4)= c_1 e_3 + c_2 e_4, \]
with  $a_1 b_2-a_2 b_1 \neq 0$, $c_2 \neq 0$, $a_3, a_4, b_3, b_4, c_1 \in \mathbb R$.  
\newline
\noindent    
Let $H$ be a subgroup of $G$ which is isomorphic to $\mathbb R$ and which is not normal in $G$. Then $H$ has the form  
$H= \exp t(\alpha e_1 + \beta e_2 + \gamma e_3 + \delta e_4)$, $t \in \mathbb R$, with  
$\alpha ^2 + \beta ^2 =1$. We can change $H$ by an automorphism of $G$ such that $H$ has the form 
$H= \exp t e_2=\{ g(0,v,0,0); v \in \mathbb R \}$. 
Since all elements of $G$ have a unique decomposition $g(x,0,y,z) g(0,v,0,0)$ 
the continuous function $v(x,y,z)$ determines a continuous section $\sigma : G/H \to G$ 
defined  by $g(x,0,y,z)H \mapsto g(x,v(x,y,z),y,z)$,  
\noindent 
where $H=\{ g(0,t,0,0); t \in \mathbb R \}$. The set $\sigma (G/H)$ acts sharply transitively on the factor space $G/H$ precisely if for all triples  
$(x_1,y_1,z_1)$, $(x_2,y_2,z_2) \in \mathbb R^3$ the equation 
\begin{equation} \label{equfil1uj} g(x,v(x,y,z),y,z) g(x_1,0,y_1,z_1)= g(x_2,0,y_2,z_2) g(0,t,0,0)  \end{equation}
has a unique solution $(x,y,z) \in \mathbb R^3$, where $g(0,t,0,0)$ is a suitable element of $H$. 
\newline
From (\ref{equfil1uj}) we obtain the equations $x=x_2-x_1$, $z=z_2- z_1$, $t=v(x,y,z)$, 
$0 = y+y_1-y_2-x_1 v(x,y,z)$.  
\noindent 
Hence equation (\ref{equfil1uj}) has a unique solution if and only if 
for every $x_0=x_2-x_1$, $z_0=z_2-z_1$ and $x_1 \in \mathbb R$ the function 
$f:y \mapsto y-x_1 v(x_0,y,z_0): \mathbb R \to \mathbb R$ is a bijective mapping. 
Let be $\psi _1 < \psi _2 \in \mathbb R$. Then $f(\psi _1) \neq f(\psi _2)$ and we may assume that  
$f(\psi _1) < f(\psi _2)$. 
We consider the inequality 
\begin{equation}  0 < f(\psi _2)-f(\psi _1) = \psi _2 - \psi _1 - x_1( v(x_0, \psi _2, z_0)- v(x_0, \psi _1, z_0))  
\nonumber  \end{equation}  
as a linear function of $x_1 \in \mathbb R$. If $v(x_0, \psi _2, z_0) \neq v(x_0, \psi _1, z_0)$, then there exists a 
$x_1 \in \mathbb R$ such that $f(\psi _2)-f(\psi _1)=0$ which is a contradiction. Hence the function $v(x,y,z)$ does not depend on the variable $y$. In this case $f$ is a monotone function and every continuous function $v(x,z)$ with $v(0,0)=0$ 
determines a loop multiplication.   
\newline
\noindent
The multiplication $(x_1,y_1,z_1) \ast (x_2,y_2,z_2)$ of the loop $L$ in the coordinate system 
$(x,y,z) \mapsto g(x, 0, y, z)H$ is determined if we apply $\sigma (g(x_1,0,y_1,z_1)H)=$ $g(x_1,v(x_1,z_1),y_1,z_1)$ 
to the left coset $g(x_2,0,y_2,z_2)H$ and find in the image coset the element of $G$ which lies in the set 
$\{ g(x,0,y,z)H; \ x,y,z \in \mathbb R \}$. A direct computation gives the multiplication (\ref{multiplileft2}) in the assertion (i). 

\noindent
This loop is proper precisely if $\sigma (G/H)=\{ g(x,v(x,z),y,z); \ x,y,z \in \mathbb R \}$ 
generates the group $G$.    
The set $\sigma (G/H)$ contains the subset $F=\{ g(x,v(x,0),0,0); \ x \in \mathbb R \}$ and 
the subgroup  $K=\{ g(0,v(0,z),y,z); \ y,z \in \mathbb R \}$ of $G$. 
The group $\langle F \rangle $ topologically generated by the set $F$ and the group $K$ generate $G$ 
if the mapping assigning to the first component of any element of $\langle F \rangle $ its second  component is not a homomorphism 
(Lemma 8 in \cite{figula}, p. 426). 
This occurs if and only if the function $v(x,0)$ is non-linear. Now we assume that $v(x,0)= c x$, $c \in \mathbb R$. 
Then the set $\sigma (G/H)$ does not generate $G$ if and only if the function $v(0,z)$ is linear. Hence the set 
$\sigma (G/H)$ generates $G$ if the function $v(x,z)$ 
does not satisfy the identities $v(x,0)= c x$ and $v(0,z)= d z$, $c,d \in \mathbb R$, simultaneously.

\noindent 
The right translation $\rho _{(a,b,c)}$ of the loop $L_v$ is the map 
\begin{equation} \label{right1equharom} \rho _{(a,b,c)}: (x,y,z) \mapsto (x,y,z) \ast (a,b,c)= (x+a, y+b -a v(x,z), z+c).  \nonumber \end{equation} 
Its inverse map $\rho _{(a,b,c)}^{-1}$ is given by 
\begin{equation} \label{right2equvketto} \rho _{(a,b,c)}^{-1}: (x,y,z) \mapsto (x-a, y-b+a v(x-a,z-c), z-c). \nonumber \end{equation} 
Since $\rho_{(0,d_1,e_1)} \rho_{(0,d_2,e_2)}= \rho_{(0,d_1+d_2,e_1+e_2)}$ and one has  
\begin{equation} \label{right3equvharom} \rho _{(a,b,c)} \rho_{(0,d,e)} \rho _{(a,b,c)}^{-1} = 
\rho _{(0,d+ a (v(x-a,z-c)- v(x-a,z-c+e)),e)}  \nonumber \end{equation}
the group $G_{\rho }$ generated by the right translations of $L_v$ has a 
normal subgroup  $N_{\rho }=\{ \rho_{(0,d,e)}; d,e \in \mathbb R \}$ isomorphic to $\mathbb R^2$. As 
 $\rho _{(a,b,c)} = \rho _{(0,b,c)} \rho _{(a,0,0)}$ one has $G_{\rho }= N_{\rho } \Sigma $, 
where $\Sigma $ is the group generated by the set $\{ \rho _{(a,0,0)}; a \in \mathbb R \}$. The Lie group $G_{\rho }$ and hence the multiplication group $Mult(L_v)$ has finite dimension if and only if the group $\Sigma $ is a finite dimensional 
Lie group. Since 
$\rho _{(a_2,0,0)} \rho _{(a_1,0,0)}: (x,y,z) \mapsto (x+a_1+a_2,y-a_1 v(x,z)-a_2 v(x+a_1,z),z)$ the group $\Sigma $ is 
the transformation group consisting of the maps 
\begin{equation} \label{right3equvnegy} \gamma _{(t,\alpha _i, \beta _i)}: (x,y,z) \mapsto 
(x+t, y+ \sum \alpha _i v(x+ \beta _i,z), z), \nonumber \end{equation}
where $t, \alpha _i, \beta _i \in \mathbb R$. By Proposition \ref{exponentialpolynomial} the group $\Sigma $ is a finite dimensional Lie group precisely if the function $v(x,z)$ is a finite linear combination of functions given by (\ref{sincosexp}).

\noindent 
Now we prove that the function $v(x,z)=v(x)$ is a polynomial of degree $n$ precisely if the multiplication group $Mult(L_v)$ of the loop $L_v$ is the direct product ${\mathcal F}_{n+2} \times \mathbb R$ with $n \ge 2$. 
Let $K$ be the group of matrices (\ref{pmatrixujuj}) with $n \ge 1$ and let $S$ be the subgroup 
$\{ g(0,t_1,t_2,\cdots ,t_n,0,0); \ t_i \in \mathbb R, i=1,2, \cdots ,n \}$. 
Then the group $K$ is isomorphic to ${\mathcal F}_{n+2} \times \mathbb R$, $n \ge 1$.      
The set $\Lambda _{v}=\{ \lambda _{(x,y,z)}; \ (x,y,z) \in L_{v} \}$ 
of all left translations of the loop $L_{v}$ defined by (\ref{multiplileft2}) in the group $K$ is   
$\Lambda _{v}=\{ g(x,v(x,z),0, \cdots , 0, y,z); \ x,y,z \in \mathbb R \}$.   
An arbitrary transversal $T$ of the group $S$ in the group $K$ has the form 
\[ T=\{ g(k,h_1(k,l,m), \cdots ,h_n(k,l,m),l, m);\  k,l,m \in \mathbb R \}, \]
where $h_j: \mathbb R^3 \to \mathbb R$, $j=1, \cdots ,n,$ are continuous  functions with $h_j(0,0,0)=0$. 
According to Lemma \ref{kepkalemma}, the group $K$ is isomorphic to the multiplication group $Mult(L_{v})$ of the loop $L_{v}$ precisely if the set 
$\{a^{-1} b^{-1} a b;\ a \in T, b \in \Lambda _{v} \}$ is contained in $S$   
and the set $\{ \Lambda _{v}, T \}$ generates the group $K$.  
The products $a^{-1} b^{-1} a b$ with $a \in T$ and 
$b \in \Lambda _{v}$ are elements of $S$ if and only if the equation 
\begin{equation} \label{equequ1} k v(x,z)= \end{equation}
\begin{equation} (-1)^{n+1} \frac{x^n}{n!} h_n(k,l,m)+(-1)^{n} \frac{x^{n-1}}{(n-1)!} h_{n-1}(k,l,m)+ \cdots +(-1)^2 x h_1(k,l,m) \nonumber  \end{equation} 
\noindent
holds for all $x,z,k,l,m \in \mathbb R$. 
Since the right hand side of (\ref{equequ1}) does not depend on the variable $z$ and the left hand side does not depend on the variables $l$ and $m$ we have $v(x,z)=v(x)$ and $h_i(k,l,m)=h_i(k)$ for all $1 \le i \le n$. Hence for any $k \neq 0$ identity (\ref{equequ1}) reduces to 
\begin{equation} \label{equequ2} v(x)=(-1)^{n+1} \frac{x^n}{n!} \frac{h_n(k)}{k}+(-1)^{n} \frac{x^{n-1}}{(n-1)!} \frac{h_{n-1}(k)}{k}+ \cdots 
+(-1)^2 x \frac{h_1(k)}{k}. \nonumber \end{equation}  
If $n=1$ for any $x \neq 0$, $k \neq 0$ we obtain 
$\frac{v(x)}{x}= \frac{h_1(k)}{k}=c= \hbox{const} (\neq 0)$, and since $v(0)=0$ this implies $v(x)= c x$. In this case the real function $v$ is linear, therefore $\Lambda _{v}$ does not generate $K$ and the group $K={\mathcal F}_3 \times \mathbb R$ cannot be the multiplication group of the loop $L_{v}$.  
 
Now let $n \ge 2$. 
Since the polynomials $x, x^2, \cdots , x^n$ are linearly independent, the function $v: \mathbb R \to \mathbb R$ depends only on the variable $x$ and the functions $h_i: \mathbb R \to \mathbb R$ are independent of $x$, the last equation  is satisfied if and only if 
$h_i(k)=a_i k$ with $a_i \in \mathbb R$ for all $1 \le i \le n$. By Lemma 8 in \cite{figula}, p. 426, 
the set $\{ \Lambda _{v}, T \}$ generates the group $K$ if and only if $a_n$ is different from $0$, since then the Lie algebra of the non-commutative group generated by the set $\{ \Lambda _{v}, T \}$ contains elements of the shape $e_2+s$ with 
$s \in \langle e_3, e_4, \cdots , e_{n+2} \rangle $.  
Hence the function $v(x)$ is a polynomial of degree $n$ and the assertion (i) is proved.

\medskip
\noindent 
Now we deal with the assertion (ii). Let $G$ be the group ${\mathcal F}_{n+2} \times \mathbb R$, $ n \ge 1$, and let 
$\sigma : G/H \to G$ be the continuous section 
\begin{equation} \label{sectionujmas} g(x,0, \cdots ,0,y,z)H \mapsto g(x,v_1(x,z),v_2(x,z), \cdots ,v_n(x,z),y,z),  \end{equation} 
\noindent 
where $H=\{ g(0,t_1,t_2, \cdots ,t_n,0,0); t_i \in \mathbb R, i=1, \cdots ,n \}$ and $v_i:\mathbb R^2 \to \mathbb R$ are continuous functions with $v_i(0,0)=0$ for all $i=1, \cdots ,n$.  
The set $\sigma (G/H)$ acts sharply transitively on the factor space $G/H$ because of for all  
$(x_1,y_1,z_1)$, $(x_2,y_2,z_2) \in \mathbb R^3$ the equation 
\begin{equation} \label{equfil1} g(x,v_1(x,z),v_2(x,z), \cdots ,v_n(x,z),y,z) g(x_1,0, \cdots ,0,y_1,z_1)= \nonumber \end{equation} 
\begin{equation} g(x_2,0, \cdots ,0,y_2,z_2) g(0,t_1,t_2, \cdots ,t_n,0,0) \nonumber  \end{equation}
has a unique solution $x = x_2-x_1$, $z=z_2- z_1$, $y=y_2-y_1+x_1 v_1(x,z)+ \cdots +\frac{(-1)^{n+1}}{n!} x_1^n v_n(x,z)$ 
with the real numbers $t_i = \sum \limits_{k=i}^n (-1)^{k-i}  \frac{x_1^{k-i}}{(k-i)!} v_k(x,z)$ for all $i=1,2, \cdots ,n$.   
\noindent 
Hence the section $\sigma $ given by (\ref{sectionujmas}) determines a $3$-dimensional topological loop $L$. A direct calculation yields that the multiplication of the loop $L$ in the coordinate system 
$(x,y,z) \mapsto g(x, 0, \cdots ,0, y, z)H$ is given by (\ref{nilpotentszorzas}) in the assertion (ii).  
\newline 
\noindent 
The group topologically generated by the left translations of the loop $L$ is isomorphic to 
$G={\mathcal F}_{n+2} \times \mathbb R$ if and only if the set 
\[ \sigma (G/H)=\{ g(x,v_1(x,z),v_2(x,z), \cdots, v_n(x,z),y,z); \ x,y,z \in \mathbb R \}  \]
generates the whole group $G$.    
The set $\sigma (G/H)$ contains the subset 
\[ F=\{ g(x,v_1(x,0),v_2(x,0), \cdots ,v_n(x,0),0,0); \ x \in \mathbb R \} \] and 
the subgroup  $K=\{ g(0,v_1(0,z),v_2(0,z), \cdots ,v_n(0,z),y,z); \ y,z \in \mathbb R \}$ of $G$. 
The group $\langle F \rangle $ topologically generated by the set $F$ and the group $K$ generate $G$ 
if $\langle F \rangle $ is a non-commutative subgroup of codimension $2$ in $G$. Using Lemma 8 in \cite{figula}, p. 426, this is the case precisely if the mapping assigning to the first component of any element of $\langle F \rangle $ its $(n+1)$-th component is not a homomorphism. 
This occurs if and only if the function $v_n(x,0)$ is non-linear. Now we assume that $v_n(x,0)= c x$, $c \in \mathbb R$. Then the set $\sigma (G/H)$ 
does not generate $G$ if and only if the function $v_n(0,z)$ is linear. Hence the set $\sigma (G/H)$ generates $G$ if the function $v_n(x,z)$ 
does not satisfy the identities $v_n(x,0)= c x$ and $v_n(0,z)= d z$, $c,d \in \mathbb R$, simultaneously.   

Now let $n>1$. By Proposition 18.16 in \cite{loops}, p. 246, the Lie group $G={\mathcal F}_{n+2} \times \mathbb R$  is the group topologically generated by all translations of the loop $L$  given by the multiplication (\ref{nilpotentszorzas})  
if and only if for every $y \in L$ the map $f(y): x \mapsto y \lambda _x \lambda _y^{-1}: L \to L$ is an element of 
$H=\{ g(0,t_1, \cdots ,t_n,0,0); t_i \in \mathbb R, i=1, \cdots ,n \}$.  
This is equivalent to the condition that the mapping 
{ \begin{equation} g(x,0, \cdots ,0,y,z) H \mapsto [g(k,v_1(k,m),v_2(k,m), \cdots ,v_n(k,m),l,m)]^{-1} \cdot \nonumber \end{equation}
\begin{equation}  [g(x,v_1(x,z),v_2(x,z), \cdots ,v_n(x,z),y,z)] g(k,0, \cdots ,0,l,m) H \nonumber \end{equation} }
\noindent
has the form 
{ \begin{equation} g(x,0, \cdots ,0,y,z) H \mapsto  \nonumber \end{equation} 
\begin{equation} g(0,s_1(k,l,m), \cdots ,s_n(k,l,m),0,0) g(x,0, \cdots ,0,y,z) H \nonumber \end{equation} }
with suitable functions $s_i(k,l,m)$, $1 \le i \le n$. This gives the equation 
{ \begin{equation}  
g(x,v_1(x,z), \cdots ,v_n(x,z),y,z)g(k,t_1, \cdots ,t_n,l,m) =   \nonumber  \end{equation} 
\begin{equation} \label{equnilpo1} g(k,v_1(k,m), \cdots ,v_n(k,m),l,m)g(0,s_1(k,l,m), \cdots ,s_n(k,l,m),0,0) \cdot \nonumber \end{equation}
\begin{equation} g(x,0, \cdots ,0,y,z)  \end{equation} } 
for a suitable element $g(0,t_1, \cdots ,t_n,0,0) \in H$. 
Equation (\ref{equnilpo1}) yields 
\noindent
\begin{equation} t_i = \sum \limits_{j=i}^n (-1)^{j-i} \frac{1}{(j-i)!} [x^{j-i}(s_j(k,l,m)+ v_j(k,m))- k^{j-i} v_j(x,z)]  \nonumber \end{equation}  
for  all $i=1,2, \cdots ,n$ and 
\begin{equation}  -x (s_1(k,l,m)+v_1(k,m)) + \cdots + \frac{(-1)^{n}}{n!} x^{n} (s_n(k,l,m)+ v_n(k,m))= \nonumber \end{equation} 
\begin{equation} \label{equequnil} - k v_1(x,z)+ \frac{k^2}{2!} v_2(x,z)+ \cdots + \frac{(-1)^n}{n!} k^n v_n(x,z). \end{equation} 
\noindent
Since the right hand side of equation (\ref{equequnil}) does not depend on the variables $l$ and $m$  and the left hand side does not depend on the variable $z$ we obtain $v_i(x,z)=v_i(x)$, $s_i(k,l,m)= s_i(k)$ for all $i=1, \cdots ,n$.  
Hence equation (\ref{equequnil})  is satisfied if and only if there are continuous functions $s_i:\mathbb R \to \mathbb R$, $i=1, \cdots ,n$, depending on the variable $k$ such that for all $x,k \in \mathbb R$  equation (\ref{equtwoside}) holds. 
Putting $k=0$ into (\ref{equtwoside}) we have $-x s_1(0)+ \frac{x^2}{2!} s_2(0)+ \cdots +(-1)^n \frac{x^n}{n!} s_n(0)=0$. 
As the polynomials $x, x^2, \cdots , x^n$ are linearly independent this equation yields that $s_i(0)=0$ for all 
$1 \le i \le n$.  This proves the assertion (ii).

\medskip
\noindent
The assertion (iii) follows from the fact that if the multiplication group as well as the group topologically generated by the left translations of a loop $L_{v_1, \cdots ,v_n}$ is isomorphic to the group 
${\mathcal F}_{n+2} \times \mathbb R$, $n \ge 2$, then the functions $v_i(x,z)$, $i=1, \cdots, n$ depend only on the variable $x$ (see assertion (ii)).  
\end{proof} 

\medskip 
\noindent 
If we choose for all $1 \le i \le n$ the functions $s_i(k)=0$, then the equation (\ref{equtwoside}) in Proposition \ref{directproduct} is symmetric in the variables $x$ and $k$. Hence the multiplication (\ref{nilpotentszorzas}) with $v_i(x,z)=v_i(x)= a_i x^i$, $a_i \in \mathbb R$, $a_n \neq 0$, gives examples for $3$-dimensional topological loops $L$ such that the group topologically generated by all left translations is isomorphic to the group 
${\mathcal F}_{n+2} \times \mathbb R$ with $n >1$ and coincides with the multiplication group of $L$.

\bigskip
\noindent
Now we deal with the $5$-dimensional indecomposable simply connected nilpotent Lie group $G$ such that its $2$-dimensional centre coincides with the commutator subgroup. This group is the group topologically generated by the left translations of $3$-dimensional topological loops. We classify these loops $L$ in the case that the stabilizer $H$ of the identity $e \in L$ is a $2$-dimensional abelian subgroup of $G$. 
But the group $G$ cannot be the multiplication group of $3$-dimensional connected topological loops.    

\begin{Prop} \label{5dimleftgroup} Let $G$ be the $5$-dimensional indecomposable simply connected nilpotent Lie group such that its $2$-dimensional centre coincides with its  commutator subgroup. Up to automorphisms of $G$ each $2$-dimensional abelian subgroup $H$ of $G$ which does not contain any non-trivial normal subgroup of $G$ has the form 
$H=\{ g(0,k,0,l,0); k,l \in \mathbb R \}$. Every $3$-dimensional connected simply connected topological loop $L$ having $G$ 
as the group topologically generated by the left translations of $L$ and $H$ as the stabilizer of $e \in L$ in $G$ is given by the multiplication 
\begin{equation} \label{nilpotentszorzasmas} (x_1,y_1,z_1) \ast (x_2,y_2,z_2)= \nonumber \end{equation} 
\begin{equation} (x_1+x_2, y_1+y_2-x_2 v_1(x_1,y_1,z_1), z_1+z_2-x_2 v_2(x_1,y_1,z_1)), \end{equation} 
where $v_i: \mathbb R^3 \to \mathbb R$, $i=1,2$, are continuous functions with $v_i(0,0,0)=0$ such that for all $x_1, x_2, y_1, y_2, z_1, z_2 \in \mathbb R$ the equations   
\begin{equation} \label{nil5dim2} 0= y+ y_1- y_2 - \frac{1}{2} (x_2+x_1) v_1(x_2-x_1, y, z)  \end{equation}
\begin{equation} \label{nil5dim3} 0= z+ z_1- z_2 - x_1 v_2(x_2-x_1, y, z)  \end{equation}
have a unique solution $(y, z) \in \mathbb R^2$ and none of the functions $v_i(x,y,z)$ satisfies the identities 
$v_i(x,0,0)=a x$ and $v_i(0,y,z)= b y+ c z$, $a,b,c \in \mathbb R$, simultaneously. 
\end{Prop} 
\begin{proof} 
According to \cite{handbook}, p. 651, the Lie algebra ${\bf g}$ of $G$ is given by the basis $\{ e_1, e_2, e_3, e_4, e_5 \}$ with 
$[e_1, e_2]=e_3$, $[e_1, e_4]=e_5$. Using the Campbell-Hausdorff series (cf. \cite{campbell}) the Lie group $G$ can 
be represented on $\mathbb R^5$ by the multiplication 
\begin{equation} \label{nil1} g(x_1,x_2,x_3,x_4,x_5) g(y_1,y_2,y_3,y_4,y_5)=  \end{equation} 
\begin{equation} g(x_1+y_1, x_2+y_2, x_3+y_3+ \frac{1}{2}(x_1 y_2- x_2 y_1), x_4+y_4, x_5+y_5+ \frac{1}{2}(x_1 y_4- x_4 y_1)). \nonumber \end{equation} 
\noindent  
The subgroup $\{\exp (t e_3+ u e_5); t, u \in \mathbb R \}= \{ g(0,0,x_3,0,x_5); x_3, x_5 \in \mathbb R \}$ is the centre $Z(G)$ and also the commutator subgroup of $G$. 
Let $H$ be a $2$-dimensional abelian subgroup of $G$ which does not contain any non-trivial normal subgroup of $G$. Then the Lie algebra ${\bf h}$ of $H$ is given by 
\begin{equation} {\bf h}= \langle e_2+ c_1 e_3+ d_1 e_5, e_4+ c_2 e_3+ d_2 e_5 \rangle, \ c_i, d_i \in \mathbb R, 
\ i=1,2. \nonumber \end{equation} 
The automorphism $\alpha $ of the Lie algebra ${\bf g}$ given by $\alpha (e_1)=e_1$, $\alpha (e_3)= e_3$, $\alpha (e_5)= e_5$, 
$\alpha (e_2)=e_2 - c_1 e_3- d_1 e_5$, $\alpha (e_4)=e_4- c_2 e_3- d_2 e_5$ 
allows us to assume that the group $H$ has the form 
$H=\{ \exp (a_2 e_2+ a_4 e_4); a_2, a_4 \in \mathbb R \}= $$ \{ g(0,k,0,l,0); k, l \in \mathbb R \}$.  
Since every element of $G$ has a unique decomposition 
$g(x, 0, y, 0, z) g(0, k, 0, l, 0)$ the continuous functions $v_i(x,y,z)$, $i=1,2$, determine a continuous section 
$\sigma : G/H \to G$, $g(x, 0, y, 0, z) H \mapsto $ 
\begin{equation} \label{section5dim} 
 g(x, v_1(x,y,z), y +\frac{1}{2} x v_1(x,y,z), v_2(x,y,z), z+ \frac{1}{2} x v_2(x,y,z)). \end{equation}     
\noindent
The set $\sigma (G/H)$ acts sharply transitively on the factor space $G/H$ precisely if for every $x_1, y_1, z_1, x_2, y_2, z_2 \in \mathbb R$ there exists a unique triple $(x,y,z) \in \mathbb R^3$ such that 
\begin{equation} \label{nil5dim4}   g(x, v_1(x,y,z), y+ \frac{1}{2} x v_1(x,y,z), v_2(x,y,z), z+ \frac{1}{2} x v_2(x,y,z)) g(x_1, 0, y_1, 0, z_1) \nonumber  \end{equation}
\begin{equation} =g(x_2, t_1, y_2 +\frac{1}{2} x_2 t_1, t_2, z_2+\frac{1}{2} x_2 t_2) \end{equation} 
for  suitable $t_1, t_2 \in \mathbb R$. Equation (\ref{nil5dim4}) yields that $x= x_2- x_1$, $t_1= v_1(x,y,z)$, 
$t_2= v_2(x,y,z)$ and equations (\ref{nil5dim2}) and (\ref{nil5dim3}) in the assertion must have a unique solution 
$(y,z) \in \mathbb R^2$. 
A direct calculation shows that in the coordinate system $(x,y,z) \mapsto g(x,0,y,0,z)H$ 
the multiplication of the loop $L_{v_1,v_2}$ corresponding to the continuous functions $v_i(x,y,z)$, $i=1,2$, can be expressed as (\ref{nilpotentszorzasmas}) in the assertion.  

The loop $L$ is proper if and only if the image set $\sigma (G/H)$ of the section $\sigma $ given by (\ref{section5dim}) 
generates the group $G$.   
The set $\sigma (G/H)$ contains the subset 
\[ F=\{ g(x,v_1(x,0,0), \frac{1}{2} x v_1(x,0,0), v_2(x,0,0), \frac{1}{2} x v_2(x,0,0)); \ x \in \mathbb R \} \] and 
the abelian normal group  $K=\{ g(0,v_1(0,y,z),y,v_2(0,y,z), z); \ y, z \in \mathbb R \}$ of $G$. 
The group $\langle F \rangle $ topologically generated by the set $F$ and the group $K$ generate $G$ if  
$v_i(x,0,0) \neq a_i x$, for $i=1,2$, $a_1,a_2 \in \mathbb R$. If there is $i \in \{ 1,2 \}$ such that the function $v_i(x,0,0)=a x$, $a \in \mathbb R$, then 
$\sigma (G/H)$ does not generate $G$ if and only if $v_i(0,y,z)=b y + c z$, $b,c \in \mathbb R$ and the assertion follows.   \end{proof}

\begin{Prop} There is no  $3$-dimensional proper connected topological loop $L$ having the $5$-dimensional 
Lie group $G$ defined by the multiplication (\ref{nil1}) in Proposition \ref{5dimleftgroup} as its multiplication group.  
\end{Prop} 
\begin{proof} 
If there exits a loop $L$ with the desired properties, then there is also a simply connected loop with such properties. By Lemma \ref{simplyconnected} we may assume that $L$ is homeomorphic to $\mathbb R^3$.  If the group $G^{(l)}$ topologically generated by all left translations of $L$ coincides with the multiplication group $Mult(L)$ of $L$ and both groups are isomorphic to the group $G$ of the assertion, then the inner mapping group $Inn(L)$ of $L$ coincides with the group 
$H= \{ g(0,k,0,l,0); k, l \in \mathbb R \}$ of Proposition \ref{5dimleftgroup}. Moreover, the multiplication of the loop $L$ is given by (\ref{nilpotentszorzasmas}). 
The set 
$\Lambda _{v_1,v_2} =\{ \lambda _{(x,y,z)}; \ (x,y,z) \in L \}$ of all left translations of the loop $L$ in the group $G$ has the form 
\begin{equation} \{ g(x, v_1(x,y,z), y +\frac{1}{2} x v_1(x,y,z), v_2(x,y,z), z+ \frac{1}{2} x v_2(x,y,z)); x, y, z \in \mathbb R \}. \nonumber \end{equation}
An arbitrary transversal $T$ of the group $Inn(L)$ in $G$ has the form 
\[ T=\{ g(m, h_1(m,n,u), n, h_2(m,n,u), u); m, n, u \in \mathbb R \}, \] 
where $h_i:\mathbb R^3 \to \mathbb R$, $i=1,2$, are continuous functions with $h_i(0,0,0)=0$. By Lemma \ref{kepkalemma} the group $G$ is isomorphic 
to $Mult(L)$ precisely if the set 
$\{a^{-1} b^{-1} a b;\ a \in \Lambda _{v_1,v_2}, b \in T \}$ is contained in $Inn(L)$  
and the set $\{ \Lambda _{v_1,v_2}, T \}$ generates the group $G$.   
The products $a^{-1} b^{-1} a b$ with $a \in \Lambda _{v_1,v_2}$ and $b \in T$ are elements of $Inn(L)$ if and only if the equations
\begin{equation} m v_1(x,y,z)= x h_1(m,n,u), \ \ m v_2(x,y,z)= x h_2(m,n,u) \nonumber \end{equation} 
are satisfied for all $m,n,u,x,y,z \in \mathbb R$. Since the right-hand side of both equations are independent of $y$ and $z$ and the left-hand side does not depend on $n$ and $u$ we obtain $v_i(x,y,z)=v_i(x)$, $h_i(m,n,u)=h_i(m)$ and hence 
$m v_i(x)= x h_i(m)$ for all $i=1,2$. 
This yields for any $m \neq 0$, $x \neq 0$ that 
\begin{equation} \label{nil5dim11uj} \frac{v_1(x)}{x}= \frac{h_1(m)}{m}= c_1, \ \  \frac{v_2(x)}{x}= \frac{h_2(m)}{m}= c_2, \nonumber \end{equation}  
where $c_1, c_2 \in \mathbb R \backslash \{ 0 \}$. Since $v_1(0)=v_2(0)=0=h_1(m)=h_2(m)$ this implies $v_i(x)=c_i x$, $h_i(m)=c_i m$, $i=1,2$.  
But then the set $\{ \Lambda _{v_1,v_2}, T \}$ does not generate 
the group $G$. Hence there is no topological loop $L$ homeomorphic to $\mathbb R^3$ such that the multiplication group $Mult(L)$ of $L$ as well as the group $G^{(l)}$ topologically generated by all left translations of $L$ coincide with the group $G$.

\medskip
\noindent
Now we assume that there is a topological loop $L$ homeomorphic to $\mathbb R^3$ such that the group $G^{(l)}$ topologically generated by the left translations of $L$ has dimension $4$ and the multiplication group $Mult(L)$ of $L$ is isomorphic to the group $G$. 
As the group $G$ has a $2$-dimensional centre, according to Lemma \ref{innermappinglemma} (b) the inner mapping group $Inn(L)$ of $L$ is a $2$-dimensional abelian subgroup of $G$. By Proposition \ref{5dimleftgroup} we may assume that the $2$-dimensional abelian subgroup $Inn(L)$ of $G$ has the form $Inn(L)=\{ g(0,k,0,l,0); k, l \in \mathbb R \}$. 
Since $G^{(l)} < G$ and there is no subgroup of $G$ isomorphic to the filiform Lie group ${\mathcal F}_{4}$ the group $G^{(l)}$ must be the direct product  ${\mathcal F}_{3} \times \mathbb R$. Each loop $L$ homeomorphic to $\mathbb R^3$ and having the group ${\mathcal F}_{3} \times \mathbb R$ as the group generated by the left translations are isomorphic to a loop $L_v$ given in Proposition \ref{directproduct} (i). We prove that none of these loops $L_v$ has the group $G$ as 
its multiplication group $Mult(L_v)$. We may assume that the set 
$\Lambda _v =\{ \lambda _{(x,y,z)}; \ (x,y,z) \in L_v \}$ of all left translations of $L_v$ in the group $G$ has the form 
\begin{equation} \Lambda _v =\{ g(x, v(x,z), y + \frac{1}{2} x v(x,z), 0, z); x, y, z \in \mathbb R \}. \nonumber \end{equation}
An arbitrary transversal $T$ of the group $Inn(L)$ in $G$ has the form 
\[ T=\{ g(m, h_1(m,n,u), n, h_2(m,n,u), u); m, n, u \in \mathbb R \}, \]  
where $h_i:\mathbb R^3 \to \mathbb R$, $i=1,2$, are continuous functions with $h_i(0,0,0)=0$. By Lemma \ref{kepkalemma} the group $G$ is isomorphic 
to $Mult(L_{v})$ precisely if the set 
$\{a^{-1} b^{-1} a b;\ a \in \Lambda _v, b \in T \}$ is contained in $Inn(L)$  
and the set $\{ \Lambda _{v}, T \}$ generates the group $G$.   
The products $a^{-1} b^{-1} a b$ with $a \in \Lambda _{v}$ and $b \in T$ are elements of $Inn(L)$ if and only if the equations
\begin{equation} m v(x,z)=x h_1(m,n,u), \ \ \ x h_2(m,n,u)=0 \nonumber \end{equation} 
are satisfied for all $m,n,u,x,z \in \mathbb R$. The last equation gives $h_2(m,n,u)=0$. But then the set $\{ \Lambda _{v}, T \}$ does not generate 
the group $G$. This contradiction gives the assertion.   \end{proof}

\section{$3$-dimensional topological loops with $1$-dimensional centre} 

In this section we show that the $3$-dimensional connected simply connected topological loops $L_{v(x,z)}$ with multiplication (\ref{multiplileft2}) given in Proposition \ref{directproduct} (i) have the direct product ${\mathcal F}_{n+2} \times _Z {\mathcal F}_{m+2}$ of the elementary filiform Lie groups ${\mathcal F}_{n+2}$ and ${\mathcal F}_{m+2}$ with amalgamated centre $Z$
 as the multiplication group if and only if the continuous function $v(x,z) $ occuring in (\ref{multiplileft2}) is a polynomial in two variables the degree of which as polynomial in $x$, respectively in $y$ is $n$, respectively $m$, where 
 $(n, m) \neq (1,1)$. For these loops the group topologically generated by the left translations is isomorphic to ${\mathcal F}_{3} \times \mathbb R$. Here we give for any $m >1$, $n>1$ examples of $3$-dimensional topological loops $L$ for which the group topologically generated by the left translations is isomorphic to the group  
${\mathcal F}_{n+2} \times _Z {\mathcal F}_{m+2}$ and coincides with the multiplication group of $L$.

\begin{Prop} \label{directproductsfiliform} 
(i) Let $L_v$ be a connected simply connected topological loop $L_v=(\mathbb R^3, \ast )$ defined by the multiplication (\ref{multiplileft2}) in Proposition \ref{directproduct}. The multiplication group $Mult(L_v)$ of the loop $L_v$ is isomorphic to the direct product 
${\mathcal F}_{n+2} \times _Z {\mathcal F}_{m+2}$ of the elementary filiform Lie groups ${\mathcal F}_{n+2}$ and 
${\mathcal F}_{m+2}$ with amalgamated centre $Z$ precisely if the function $v(x,z)$ has the form 
$v(x,z)=a_n x^n+ \cdots +a_1 x+ b_m z^m+ \cdots + b_1 z$, 
where $(n, m) \in \mathbb N^2 \backslash \{(1,1) \}$, $a_i, b_j \in \mathbb R$, $i=1, \cdots ,n$, $j=1, \cdots ,m$, 
$a_n \neq 0, b_m \neq 0$.    
\newline
\noindent
(ii) Let $L$ be a connected simply connected topological loop $L=(\mathbb R^3, \ast )$ with the multiplication 
{\small \begin{equation} \label{amalgamatedfiliform} (x_1,y_1,z_1) \ast (x_2,y_2,z_2)=  \nonumber \end{equation} 
\begin{equation} (x_1+x_2, y_1+y_2, z_1+z_2- x_2 v_1(x_1,y_1)+  \frac{x_2^2}{2!} v_2(x_1,y_1)+ \cdots + \frac{(-1)^n}{n!} x_2^n  v_n(x_1,y_1) - \nonumber \end{equation} \begin{equation}  
-y_2 u_1(x_1,y_1)+  \frac{y_2^2}{2!} u_2(x_1,y_1)+ \cdots + \frac{(-1)^m}{m!} y_2^m  u_m(x_1,y_1)),   \end{equation} }
\noindent
where $v_i: \mathbb R^2 \to \mathbb R$, $u_j: \mathbb R^2 \to \mathbb R$, $i=1, 2, \cdots ,n$, $j=1,2, \cdots ,m$,  are continuous functions with 
$v_i(0,0)=0$, $u_j(0,0)=0$ such that 
the functions $v_n$ as well as $u_m$ do not fulfill the identities $v_n(x,0)=a_1 x$, $v_n(0,y)=a_2 y$, $u_m(x,0)=a_3 x$, $u_m(0,y)=a_4 y$, 
$a_i \in \mathbb R$, $i=1,2,3,4$, simultaneously. 
Then the group $G$ topologically generated by the left translations of $L$ is isomorphic to the group ${\mathcal F}_{n+2} \times _Z {\mathcal F}_{m+2}$ with $n,m \ge 1$. 

For $n >1$ and $m >1$ the group $G$ coincides with the group $Mult(L)$ topologically generated by all left and right  translations 
of $L$ if and only if there are continuous functions $s_i: \mathbb R^2 \to \mathbb R$, $f_j: \mathbb R^2 \to \mathbb R$ 
with $s_i(0,0)=0$ and $f_j(0,0)=0$, $i=1, \cdots ,n$, $j=1, \cdots ,m$, such that for all 
$x,y,p,q \in \mathbb R$ one has 
\begin{equation} \label{equequamalgamated} -x (s_1(p,q)+v_1(p,q))+ \cdots +(-1)^n \frac{x^n}{n!} (s_n(p,q) + v_n(p,q))- \nonumber \end{equation} 
\begin{equation} - y (f_1(p,q)+u_1(p,q))+ \cdots +(-1)^m \frac{y^m}{m!} (f_m(p,q) + u_m(p,q))= \nonumber \end{equation}
\begin{equation} =-p v_1(x,y)+ \frac{p^2}{2!} v_2(x,y)+ \cdots +(-1)^n \frac{p^n}{n!} v_n(x,y)- \nonumber \end{equation}
\begin{equation} -q u_1(x,y)+ \frac{q^2}{2!} u_2(x,y)+ \cdots +(-1)^m \frac{q^m}{m!} u_m(x,y).  \end{equation}  
\end{Prop} 
\begin{proof} 
We can represent the elements of the Lie group ${\mathcal F}_{n+2} \times _Z {\mathcal F}_{m+2}$ as follows:  
{\small \begin{equation} \label{amalgamated} g(c,a_1,a_2, \cdots , a_n, b, d_1, d_2, \cdots ,d_m, k_1+k_2)=  \end{equation}
\begin{equation}  \label{pmatrix} \begin{pmatrix} 
1 & a_1 & a_2  & \dots & a_{n-1} & a_n & k_1 \\
0 & 1 & 0  & 0 & \dots  & 0 & -c  \\
0 & -c & 1 & 0 & \dots  & 0 & \frac{c^2}{2!}  \\
0 & \frac{c^2}{2!}  & -c & 1 & \dots & 0 & - \frac{c^3}{3!}  \\
\hdotsfor{7} \\
0 & \frac{(-1)^{n-1}}{(n-1)!} c^{n-1}  & \frac{(-1)^{n-2}}{(n-2)!}  c^{n-2} & \dots  & \frac{(-1)}{(n-(n-1))!} c^{1}  & 1 & \frac{(-1)^n}{n!} c^{n} \\
0 & 0  & 0 & \dots  &  0 & 0 & 1 \end{pmatrix}, \nonumber \end{equation}
\begin{equation} \label{pmatrixuj} \begin{pmatrix} 
1 & d_1 & d_2  & \dots & d_{m-1} & d_m & k_2 \\
0 & 1 & 0  & 0 & \dots  & 0 & -b  \\
0 & -b & 1 & 0 & \dots  & 0 & \frac{b^2}{2!} \\
0 & \frac{b^2}{2!}  & -b & 1 & \dots & 0 & - \frac{b^3}{3!} \\
\hdotsfor{7} \\
0 & \frac{(-1)^{m-1}}{(m-1)!} b^{m-1}  & \frac{(-1)^{m-2}}{(m-2)!}  b^{m-2} & \dots  & \frac{(-1)}{(m-(m-1))!} b^{1}  & 1 & \frac{(-1)^m}{m!} b^m \\
0 & 0  & 0 & \dots  &  0 & 0  & 1 \end{pmatrix}, k_1+k_2, \nonumber \end{equation} }
\noindent
with $a_i, b, c, d_j, k_1, k_2 \in \mathbb R, i=1,2, \cdots , n$, $j=1,2, \cdots , m$. 
\newline
\noindent
Let $K$ be the group ${\mathcal F}_{n+2} \times _Z {\mathcal F}_{m+2}$, $n, m \ge 1$ and $S$ be the subgroup 
\[ S=\{ g(0,t_1, \cdots ,t_n,0,k_1, \cdots ,k_m,0); \ t_i, k_j \in \mathbb R, i=1,2, \cdots ,n, j=1, \cdots ,m \}. \]
The set $\Lambda _{v}=\{ \lambda _{(x,y,z)}; \ (x,y,z) \in L_{v} \}$ 
of all left translations of the loop $L_{v}$ defined by multiplication (\ref{multiplileft2}) in the group $K$ has the shape  
\[ \Lambda _{v}=\{ g(x,v(x,z),0, \cdots , 0, y,0, \cdots ,0,z); \ x,y,z \in \mathbb R \}.  \]  
An arbitrary transversal $T$ of the group $S$ in the group $K$ has the form 
\[ \{ g(k,h_1(k,l,w), \cdots ,h_n(k,l,w),l,f_1(k,l,w), \cdots ,f_m(k,l,w), w);\  k,l,w \in \mathbb R \}, \]
where $h_i: \mathbb R^3 \to \mathbb R$, $f_j: \mathbb R^3 \to \mathbb R$, $i=1, \cdots ,n$, $j=1, \cdots ,m$, are continuous  functions with $h_i(0,0,0)=f_j(0,0,0)=0$. 
According to Lemma \ref{kepkalemma}, the group $K$ is isomorphic to the multiplication group $Mult(L_{v})$ of the loop $L_{v}$ precisely if the set 
$\{a^{-1} b^{-1} a b;\ a \in T, b \in \Lambda _{v} \}$ is contained in $S$   
and the set $\{ \Lambda _{v}, T \}$ generates the group $K$.  
The products $a^{-1} b^{-1} a b$ with $a \in T$ and 
$b \in \Lambda _{v}$ are elements of $S$ if and only if the equation 
\begin{equation} \label{equamalgamated1} k v(x,z)= \end{equation}
\begin{equation} (-1)^{n+1} \frac{x^n}{n!} h_n(k,l,w)+(-1)^{n} \frac{x^{n-1}}{(n-1)!} h_{n-1}(k,l,w)+ \cdots +(-1)^2 x h_1(k,l,w)+ \nonumber  \end{equation} 
\begin{equation} (-1)^{m+1} \frac{z^m}{m!} f_m(k,l,w)+(-1)^{m} \frac{z^{m-1}}{(m-1)!} f_{m-1}(k,l,w)+ \cdots +(-1)^2 z f_1(k,l,w) \nonumber  \end{equation} 
\noindent
holds for all $x,z,k,l,w \in \mathbb R$. 
Since the left hand side of (\ref{equamalgamated1}) does not depend on the variables $l$ and $w$ we have $h_i(k,l,w)=h_i(k)$ and $f_j(k,l,w)=f_j(k)$ 
for all $1 \le i \le n$, $1 \le j \le m$. Hence for any $k \neq 0$ identity (\ref{equamalgamated1}) reduces to 
\begin{equation} \label{equamalgamated2} v(x,z)=(-1)^{n+1} \frac{x^n}{n!} \frac{h_n(k)}{k}+(-1)^{n} \frac{x^{n-1}}{(n-1)!} \frac{h_{n-1}(k)}{k}+ \cdots 
+(-1)^2 x \frac{h_1(k)}{k}+ \nonumber \end{equation}  
\begin{equation} (-1)^{m+1} \frac{z^m}{m!} \frac{f_m(k)}{k}+(-1)^{m} \frac{z^{m-1}}{(m-1)!} \frac{f_{m-1}(k)}{k}+ \cdots 
+(-1)^2 z \frac{f_1(k)}{k}.  \end{equation}  
First let $n=m=1$. Then for $z=0$ and for any $x \neq 0$ we get $\frac{v(x,0)}{x}=\frac{h_1(k)}{k}=a= \hbox{const} (\neq 0)$. Moreover, for $x=0$ and for any 
$z \neq 0$ we have $\frac{v(0,z)}{z}=\frac{f_1(k)}{k}=b=\hbox{const} (\neq 0)$. Since $v(0,0)=0$ this implies $v(x,0)= a x$ and $v(0,z)=b z$. This is a contradiction to the fact that the function $v$ does not fulfill the identities $v(x,0)=a x$ and $v(0,z)= b z$, $a,b \in \mathbb R$, simultaneously (cf. Proposition \ref{directproduct} (i)).   

Now let $(n,m) \in \mathbb N^2 \setminus \{(1,1) \}$.  
Since the polynomials $x, x^2, \cdots , x^n$ as well as $z, z^2, \cdots , z^m$ are linearly independent, the function $v: \mathbb R^2 \to \mathbb R$ depends only on the variables $x$, $z$ and the functions $h_i: \mathbb R \to \mathbb R$, $f_j: \mathbb R \to \mathbb R$ are independent of $x$, $z$ the equation (\ref{equamalgamated2}) is satisfied if and only if 
$h_i(k)=a_i k$  and $f_j(k)= b_j k$ with $a_i, b_j \in \mathbb R$ for all $1 \le i \le n$, $1 \le j \le m$, $a_n \neq 0$, $b_m \neq 0$. This yields assertion (i).  

\medskip
\noindent
Now we prove assertion (ii). Let $G$ be the group ${\mathcal F}_{n+2} \times _Z {\mathcal F}_{m+2}$ and $\sigma : G/H \to G$ be the continuous section $g(x,0, \cdots ,0,y,0, \cdots ,0,z)H \mapsto $ 
\begin{equation} \label{sectionuj}  g(x,v_1(x,y), \cdots ,v_n(x,y),y,u_1(x,y),\cdots ,u_m(x,y),z), \end{equation} 
where $H= \{ g(0,t_1,t_2, \cdots ,t_n,0,k_1,k_2, \cdots ,k_m,0); t_i, k_j \in \mathbb R \}$ and $v_i: \mathbb R^2 \to \mathbb R$, $u_j: \mathbb R^2 \to \mathbb R$ are continuous functions with $v_i(0,0)= u_j(0,0)=0$ for all $i=1, \cdots ,n$, 
$j=1, \cdots ,m$.  
The set  
$\sigma (G/H)$ acts sharply transitively on the factor space $G/H$ because of  for all  
$(x_1,y_1,z_1), (x_2,y_2,z_2) \in \mathbb R^3$ the equation 
\begin{equation} \label{equfiluj2} g(x,v_1(x,y), \cdots ,v_n(x,y),y,u_1(x,y),\cdots ,u_m(x,y),z) \cdot \nonumber \end{equation}
\begin{equation} g(x_1,0, \cdots ,0,y_1,0, \cdots ,0,z_1)= \nonumber \end{equation} 
\begin{equation} g(x_2,0, \cdots ,0,y_2,0, \cdots ,0,z_2) g(0,t_1,t_2, \cdots ,t_n,0,k_1,k_2, \cdots ,k_m,0)  \end{equation}
has the unique solution $x= x_2-x_1$, $y= y_2-y_1$,
\begin{eqnarray} 
z &=& z_2-z_1+x_1 v_1(x,y)+ \cdots +\frac{(-1)^{n+1}}{n!} x_1^n v_n(x,y)+   \nonumber \\
  & & +y_1 u_1(x,y)-\frac{y_1^2}{2!} u_2(x,y)+ \cdots +\frac{(-1)^{m+1}}{m!} y_1^m u_m(x,y) \nonumber \end{eqnarray}  
with the real numbers 
\begin{eqnarray} 
t_i &=& \sum \limits_{k=i}^n (-1)^{k-i}  \frac{x_1^{k-i}}{(k-i)!} v_k(x,y) \ \hbox{for all}\ i=1,2, \cdots ,n, \nonumber \\
k_j &=& \sum \limits_{l=j}^m (-1)^{l-j}  \frac{y_1^{l-j}}{(l-j)!} u_l(x,y) \ \hbox{for all}\ j=1,2, \cdots ,m. \nonumber \
\nonumber \end{eqnarray} 
Hence the section $\sigma $ defines a $3$-dimensional topological loop $L$ the multiplication of which in the coordinate system 
$(x,y,z) \mapsto g(x, 0, \cdots ,0, y, 0, \cdots ,0,z)H$ is determined  if we apply 
\begin{equation} \sigma (g(x_1,0, \cdots ,0,y_1,0, \cdots ,0,z_1)H)= \nonumber \end{equation} 
\begin{equation} g(x_1,v_1(x_1,y_1), \cdots ,v_n(x_1,y_1),y_1,u_1(x_1,y_1), \cdots ,u_m(x_1,y_1),z_1) \nonumber \end{equation}
\noindent
to the left coset 
$g(x_2,0, \cdots ,0,y_2,0, \cdots ,0,z_2)H$ and find in the image coset the element of $G$ which lies in the set 
$\{ g(x,0, \cdots ,0,y,0, \cdots ,0,z)H; \ x,y,z \in \mathbb R\}$. A direct computation yields the multiplication (\ref{amalgamatedfiliform}) of the assertion (ii).  
\noindent
The group topologically generated by the left translations of the loop $L$ given by (\ref{amalgamatedfiliform}) is isomorphic to $G={\mathcal F}_{n+2} \times _Z {\mathcal F}_{m+2}$ if and only if the set 
$\sigma (G/H)= \{g(x,v_1(x,y), \cdots, v_n(x,y),y,u_1(x,y), \cdots, u_m(x,y),z);\ x,y,z \in \mathbb R \}$
generates the whole group $G$.    
The set $\sigma (G/H)$ contains the subset 
\[ F=\{ g(x,v_1(x,y), \cdots ,v_n(x,y),y, u_1(x,y), \cdots , u_m(x,y),0); \ x,y \in \mathbb R \} \] and 
the centre  $Z=\{ g(0, \cdots ,0, z); \ z \in \mathbb R \}$ of $G$. 
The group $\langle F \rangle $ topologically generated by the set $F$ and the group $Z$ generate $G$ 
if the projection of $\langle F \rangle $ onto the set $\{ g(k,k_1, \cdots ,k_n,l,l_1, \cdots ,l_m,0); k,l,k_i,l_j \in \mathbb R, i=1,\cdots ,n, j=1, \cdots m \}$ has dimension $n+m+2$. 
The set $F$ contains the subsets 
\[ F_1=\{ g(x,v_1(x,0), \cdots ,v_n(x,0),0, u_1(x,0), \cdots , u_m(x,0),0); \ x \in \mathbb R \}, \] 
\[ F_2=\{ g(0,v_1(0,y), \cdots ,v_n(0,y),y, u_1(0,y), \cdots , u_m(0,y),0); \ y \in \mathbb R \}. \]  
We have $F_1 \cap F_2= 0$. 
Using Lemma 8 (cf. \cite{figula}, p. 426) for the subsets $F_1$ and $F_2$ we obtain that the set $\sigma (G/H)$ generates $G$ precisely if the functions 
$v_n(x,y)$ and $u_m(x,y)$ do not satisfy the identities $v_n(x,0)=a_1 x$, $v_n(0,y)=a_2 y$, $u_m(x,0)=a_3 x$, $u_m(0,y)=a_4 y$ simultaneously, 
$a_i \in \mathbb R$, $i=1,2,3,4$.  

Now let $n>1$ and $m>1$. 
According to Proposition 18.16 in \cite{loops}, p. 246, the Lie group $G={\mathcal F}_{n+2} \times _Z {\mathcal F}_{m+2}$  coincides with the group topologically generated by all translations of the loop $L$  given by the multiplication (\ref{amalgamatedfiliform})  
if and only if for every $y \in L$ the map $f(y): x \mapsto y \lambda _x \lambda _y^{-1}: L \to L$ is an element of 
\[ H=\{ g(0,t_1, \cdots ,t_n,0,k_1, \cdots ,k_m,0); t_i, k_j \in \mathbb R, i=1, \cdots ,n, j=1, \cdots ,m \}. \] 
This is equivalent to the condition that for every $(p,q,r)$, $(x,y,z) \in \mathbb R^3$ the element 
\begin{equation} [g(p,v_1(p,q), \cdots ,v_n(p,q),q,u_1(p,q), \cdots ,u_m(p,q),r)]^{-1} \cdot \nonumber \end{equation}
\begin{equation} [g(x,v_1(x,y), \cdots ,v_n(x,y),y, u_1(x,y), \cdots ,u_m(x,y),z)] \cdot \nonumber \end{equation}
\begin{equation} g(p,0, \cdots ,0,q,0, \cdots ,0,r) \nonumber \end{equation} 
\noindent
can be represented by an element 
\begin{equation} g(0,s_1(p,q,r), \cdots ,s_n(p,q,r),0, f_1(p,q,r), \cdots ,f_m(p,q,r),0) \cdot \nonumber \end{equation}
\begin{equation} g(x,0, \cdots ,0,y,0, \cdots ,0,z) g(0,-t_1, \cdots ,-t_n,0,-k_1, \cdots ,-k_m,0) \nonumber \end{equation} 
with suitable functions $s_i(p,q,r)$, $f_j(p,q,r)$, $1 \le i \le n$, $1 \le j \le m$ and $t_1$, $\cdots $, $t_n$, $k_1$, $\cdots $, $k_m \in \mathbb R$.  This gives the following  
\begin{equation}  
g(x,v_1(x,y), \cdots ,v_n(x,y),y,u_1(x,y),\cdots ,u_m(x,y),z) \cdot \nonumber \end{equation}  
\begin{equation} g(p,t_1, \cdots ,t_n,q,k_1, \cdots ,k_m,r) =   \nonumber  \end{equation} 
\begin{equation} \label{equnilpo1uj} g(p,v_1(p,q), \cdots ,v_n(p,q),q,u_1(p,q),\cdots ,u_m(p,q),r) \cdot \nonumber \end{equation}  
\begin{equation} g(0,s_1(p,q,r), \cdots ,s_n(p,q,r),0,f_1(p,q,r), \cdots ,f_m(p,q,r),0) \cdot \nonumber \end{equation}
\begin{equation} g(x,0, \cdots ,0,y,0, \cdots ,0,z).   \end{equation} 
From (\ref{equnilpo1uj}) it follows 
\noindent
\begin{equation} t_i = \sum \limits_{k=i}^n (-1)^{k-i} \frac{1}{(k-i)!} [x^{k-i}(s_k(p,q,r)+ v_k(p,q))- p^{k-i} v_k(x,y)]  \nonumber \end{equation}  
\begin{equation} k_j = \sum \limits_{l=j}^m (-1)^{l-j} \frac{1}{(l-j)!} [y^{l-j}(f_l(p,q,r)+ u_l(p,q))- q^{l-j} u_l(x,y)]  \nonumber \end{equation}  
for  all $i=1, \cdots ,n$, $j=1, \cdots ,m$ and 
\begin{equation}  -x (s_1(p,q,r)+v_1(p,q))+ \cdots + \frac{(-1)^{n}}{n!} x^{n} (s_n(p,q,r)+ v_n(p,q))- \nonumber \end{equation}
\begin{equation}  -y (f_1(p,q,r)+u_1(p,q))+ \cdots + \frac{(-1)^{m}}{m!} y^{m} (f_m(p,q,r)+ u_m(p,q))= \nonumber \end{equation} 
\begin{equation} \label{equequniluj3} - p v_1(x,y)+ \frac{p^2}{2!} v_2(x,y)+ \cdots + \frac{(-1)^n}{n!} p^n v_n(x,y) - \nonumber \end{equation} 
\begin{equation} - q u_1(x,y)+ \frac{q^2}{2!} u_2(x,y)+ \cdots + \frac{(-1)^m}{m!} q^m u_m(x,y). \end{equation}
\noindent
Since the right hand side of equation (\ref{equequniluj3}) does not depend on the variable $r$, so is the left hand side and we have $s_i(p,q,r)= s_i(p,q)$, $f_j(p,q,r)= f_j(p,q)$ for all $1 \le i \le n$ $1 \le j \le m$.   
Using this, equation (\ref{equequniluj3}) reduces to (\ref{equequamalgamated}). Putting $y=p=q=0$ into  (\ref{equequamalgamated}) and using the fact that the polynomials $x, x^2, \cdots ,x^n$ are linearly independent we get that 
$s_i(0,0)=0$ for all $1 \le i \le n$. Putting $x=p=q=0$ into  (\ref{equequamalgamated}) and using the fact that the polynomials $y, y^2, \cdots ,y^n$ are linearly independent we get that 
$f_j(0,0)=0$ for all $1 \le j \le m$. This proves the assertion (ii). 
\end{proof} 

\medskip
\noindent
It is easy to see that if we choose for all $1 \le i \le n$ and $1 \le j \le m$ the functions $s_i(p,q)=0$, $f_j(p,q)=0$,  $v_i(x,y)=v_i(x)=a_i x^i$ and $u_j(x,y)=u_j(y)= b_j y^j$, $a_i, b_j \in \mathbb R$, $a_n \neq 0$, $b_m \neq 0$, then equation 
(\ref{equequamalgamated}) is satisfied. Hence the multiplication (\ref{amalgamatedfiliform}) with $v_i(x,y)=v_i(x)$ and $u_j(x,y)=u_j(y)$ gives examples for $3$-dimensional topological loops $L$ for which the group topologically generated by the left translations is isomorphic to the group  ${\mathcal F}_{n+2} \times _Z {\mathcal F}_{m+2}$ and coincides with the multiplication group of $L$.

Author's address: \'Agota Figula \\               
Institute of Mathematics \\ 
University of Debrecen \\
H-4010 Debrecen, P.O.B. 12, Hungary\\            
figula@math.klte.hu                


\end{document}